\documentclass[a4paper, 11pt]{article}

\usepackage{calc, amsmath, amsthm, a4, latexsym, amssymb, layout, bm, xspace, url, bbm, caption, subcaption} 
\usepackage[normalem]{ulem} 
\usepackage[dvips]{graphicx}
\usepackage{mathtools}
\usepackage{stmaryrd}
\usepackage[usenames,dvipsnames,svgnames,table]{xcolor}
\usepackage{tikz}
\usepackage{graphicx}
\usepackage{pgfplots}
\usetikzlibrary{arrows,shapes,positioning}

\mathtoolsset{showonlyrefs}

\newtheorem{thm}{Theorem}[section]
\newtheorem{theorem}[thm]{Theorem}
\newtheorem{lemma}[thm]{Lemma}

\newtheorem{prop}[thm]{Proposition}

\theoremstyle{definition}

\newtheorem{example}[thm]{Example}

\newtheorem{rem}[thm]{Remark}
\newtheorem{remark}[thm]{Remark}

\newcommand{\be}[1]{\begin{equation}\label{#1}}
\newcommand{\ee}{\end{equation}}
\newcommand{\ba}{\begin{array}}
\newcommand{\ea}{\end{array}}
\newcommand{\bal}{\begin{aligned}}
\newcommand{\eal}{\end{aligned}}

\newcommand{\R}{\mathbb{R}}

\newcommand{\N}{\mathbb{N}}

\newcommand{\Z}{\mathbb{Z}}
\newcommand{\E}{\mathbb{E}}
\newcommand{\p}{\mathbb{P}}

\newcommand{\calC}{\mathcal{C}}

\newcommand{\calM}{\mathcal{M}}

\newcommand{\cC}{\mathcal{C}}
\newcommand{\cD}{\mathcal{D}}

\newcommand{\cF}{\mathcal{F}}

\newcommand{\cI}{\mathcal{I}}
\newcommand{\cJ}{\mathcal{J}}

\newcommand{\cL}{\mathcal{L}}
\newcommand{\cM}{\mathcal{M}}

\newcommand{\cP}{\mathcal{P}}

\newcommand{\cT}{\mathcal{T}}

\newcommand{\cV}{\mathcal{V}}
\newcommand{\cW}{\mathcal{W}}

\newcommand{\eps}{\varepsilon}

\newcommand{\ssup}[1] {{\scriptscriptstyle{({#1}})}}

 \newcommand{\supp}{{\rm supp}}

\renewcommand{\rho}{\varrho}

\newcommand{\spaeter}[1]{}

\newcommand{\bE}{\ensuremath{\mathbb{E}}}

\newcommand{\bN}{\ensuremath{\mathbb{N}}}

\newcommand{\bP}{\ensuremath{\mathbb{P}}}

\newcommand{\bR}{\ensuremath{\mathbb{R}}}

\newcommand{\bT}{\ensuremath{\mathbb{T}}}

\newcommand{\bZ}{\ensuremath{\mathbb{Z}}}

\newcommand{\ind}{\ensuremath{\mathbbm{1}}}

\newcommand{\ddx}[1][1]{\ifnum#1=1 \frac{d}{dx} \else \frac{d^{#1}}{dx^{#1}} \fi}
\newcommand{\ddy}[1][1]{\ifnum#1=1 \frac{d}{dy} \else \frac{d^{#1}}{dy^{#1}} \fi}
\newcommand{\ddt}[1][1]{\ifnum#1=1 \frac{d}{dt} \else \frac{d^{#1}}{dt^{#1}} \fi}

\newcommand{\bfx}{{\mathbf x}}
\newcommand{\bfy}{{\mathbf y}}

\newcommand{\bfX}{{\mathbf X}}
\newcommand{\bfY}{{\mathbf Y}}

\begin{document}

\begin{center}
{\Large \bf Entrance laws for annihilating Brownian motions \\[1mm]
and the continuous-space voter model
}
\\[5mm]
\vspace{0.7cm}
\textsc{Matthias Hammer\footnote{Institut f\"ur Mathematik, Technische Universit\"at Berlin, Stra\ss e des 17. Juni 136, 10623 Berlin, Germany.},
Marcel Ortgiese\footnote{Department of Mathematical Sciences, University of Bath, Claverton Down, Bath, BA2 7AY,
United Kingdom.} and Florian V\"ollering\footnote{Fakult\"{a}t f\"ur Mathematik und Informatik, Universit\"at Leipzig, 
Augustusplatz 10, 04109 Leipzig, Germany.
}
} 
\\[0.8cm]
\fboxsep2mm
{\small 6 March 2019} 
\end{center}

\vspace{0.3cm}

\begin{abstract}\noindent 
Consider a system of particles moving independently as Brownian motions until two of them meet, 
when the colliding pair annihilates instantly. 
The construction of such a system of annihilating Brownian motions (aBMs) is straightforward as long as we start with a finite number of particles, but is more involved for infinitely many particles. 
In particular, if we let the set of starting points become increasingly dense in the real line it is not obvious whether the resulting systems of aBMs converge and what the possible limit points (entrance laws) are. 
In this paper, we show that aBMs arise as the interface model of the continuous-space voter model. This link allows us to provide a full classification of entrance laws for aBMs. 
We also give some examples showing how different entrance laws can be obtained via finite approximations. 
Further, we discuss the relation of the continuous-space voter model to the stepping stone and other related models.
Finally, we obtain an expression for the $n$-point densities of aBMs starting from an arbitrary entrance law.

  \par\medskip

  \end{abstract}


\section{Introduction}\label{intro}

Consider a system of particles moving independently as Brownian motions such that whenever two of them meet, the colliding pair annihilates instantly. As long as we start with a finite number of particles, 
the construction of such a system of annihilating Brownian motions (from now on called aBMs) is straightforward.
It is also possible to start aBMs from infinitely many particles, provided that the initial positions do not accumulate, i.e.\ form a \emph{discrete} and \emph{closed} {(or equivalently, locally finite)} subset of the real line. The construction of such an infinite system is already not completely trivial, see e.g.\ \cite[Sec.\ 4.1]{TZ11} or Section \ref{sec:construction} below for some details.
Thus a suitable state space for the evolution of aBMs is given by
\[\cD:=\{\bfx\subseteq\R:\bfx\text{ is discrete and closed}\},\]
and for each $\bfx\in\cD$ a system of aBMs starting from $\bfx$ can be constructed as a (strong) Markov process $\bfX^\bfx=(\bfX^\bfx_t)_{t\ge0}$ taking values in $\cD$.

Now let $\bfx_{n}\in\cD$ be a sequence of discrete closed subsets of $\R$ which eventually become dense in the real line. 
We can ask the question, for which such sequences the corresponding aBM processes $\bfX^{\bfx_n}$ converge and what the possible limit points are. 
Intuitively, such a limit should correspond to a system of aBMs `started everywhere on the real line'. 
More formally, such a limit gives rise to an \emph{entrance law} for the semigroup of aBMs on $\cD$ (see \eqref{eq:entrance_law} below where we recall the formal definition). 
However, it is not clear a priori whether all asymptotically dense sets of starting points will lead to the same entrance law.
This is in contrast to \emph{coalescing} Brownian motions (from now on called cBMs), which have a monotonicity property. 
For cBMs, it is possible to add initial particles one by one, and as long as the asymptotic set of starting points is dense one ends up with a universal maximal object, 
the \emph{Arratia flow}, see \cite{A79}.
Thus in the coalescing case there is a unique maximal entrance law, where Brownian motions are started everywhere on the real line. 
When starting cBMs in all space-time points, the resulting object is called the \emph{Brownian web}, see e.g.\ \cite{SSS17} for a recent survey.

For the annihilating case, in \cite{TZ11} Tribe and Zaboronski define a corresponding `maximal' entrance law as a `thinned' version of the maximal entrance law for cBMs
(see Sec.\ 2.1 of their paper for the well-known thinning relation linking coalescing and annihilating systems). 
Moreover, they argue that this entrance law can be approximated by aBMs started from the lattice $\frac{1}{n}\Z$, or from points of a Poisson process with intensity $n$, by sending $n\to\infty$,
but point out that the domain of attraction of this entrance law is not clear.

In this paper, we show that indeed different approximations of $\R$ by asymptotically dense sets will typically lead to different entrance laws for aBMs, as opposed to the case for cBMs. 
For example, if one starts a system of aBMs in $\bfx_{n}:=\frac1n\bZ+\{0,\frac1{n^2}\}$ so that starting points appear in close-by pairs, then typically the pairs annihilate and in the limit there are no surviving annihilating Brownian motions at all.
In our main result, Theorem \ref{thm:characterization_entrance_laws}, we will give a complete classification of entrance laws for aBMs via identification with measurable functions $u:\bR\to[0,1]$.

Our classification of the entrance laws is based on a close connection of
aBMs with the \emph{continuous-space voter model}, which is a generalization of the classical discrete voter model to a continuous space setting. 
We will review this model and some of the relevant literature in Section \ref{sec:cSSM}.

As an application of this relation and the technique of duality, we can compute $n$-point densities for aBMs, i.e.\ the probability density of finding $n$ particles at given points.
There has been some interest in these $n$-point densities recently, and indeed the main result in ~\cite{TZ11} is to show that a system of cBMs, but  also of aBMs, started from the `maximal' entrance law forms a Pfaffian point process and to give an expression for the densities. In particular, \cite{TZ11}  show that these expressions can be used to derive large-time asymptotics.

In contrast, our result allows us to calculate  $n$-point densities for any entrance law. For example, we can explicitly compute the $1$-particle density function and compare it to the density function under the entrance law constructed in  \cite{TZ11}.
We can show that the latter is only maximal when compared to homogeneous entrance laws, so that a more appropriate name would be `maximal  homogeneous'. 
Our technique also gives an expression for $n$-point densities with $n >1$, which is however less explicit.

The paper is structured as follows: In Section~\ref{sec:classification}, we will state 
{our results, namely the classification of entrance laws for aBMs and the corresponding $n$-point densities.
}
In Section~\ref{sec:cSSM}, we will explain the connection to the continuous-space voter model. 
We use the relation to the voter model and its duality in Section~\ref{sec:proofs} to prove the results
of 
Section~\ref{sec:classification}. 
In the appendix, we recall in Section~\ref{sec:construction} how to construct aBMs starting from an infinite discrete closed set
and prove {two technical results} in Section~\ref{sec:technicalities}.

\subsection{Notation and preliminaries}\label{ssec:notation}
The following notation and definitions will be used throughout the paper: Recall that $\cD$ denotes the space of discrete closed subsets of $\R$,
{
and that we write $\bfx$ for a generic element of $\cD$. 
With slight abuse of notation, we will occasionally 
use the same symbol for vectors and write also $\bfx=(x_1,\ldots,x_n)\in\R^n$. 
We denote by $\bR^{n,\uparrow}:=\{\bfx\in\bR^n: x_1<\cdots<x_n\}$ resp.\ $\bR^{n,\downarrow}:=\{\bfx\in\bR^n: x_1>\cdots>x_n\}$ the space of increasing resp.\ decreasing vectors in $\bR^n$.

Moreover, recall that we denote by $\bfX^\bfx=(\bfX_t^\bfx)_{t\ge0}$ 
a countable system of annihilating Brownian motions starting from $\bfx\in\cD$.
See e.g.\ \cite[Sec.\ 4.1]{TZ11} and Section \ref{sec:construction} below for two possible approaches to the construction of $\bfX^\bfx$ in case that the initial condition $\bfx$ is countably infinite. 
Considering this system as a (strong) Markov process with state space $\cD$, 
we write also $\bfX=(\bfX_t)_{t\ge0}$ for the canonical process on the path space $\cD^{[0,\infty)}$ and $(\p_\bfx)_{\bfx\in\cD}$ for the corresponding family of probability measures 
such that $\bfX$ starts from $\bfx\in\cD$ under $\p_\bfx$. 
The corresponding Markov semigroup on $\cD$ will be denoted by $(P_t)_{t\ge0}$. 
Note that we did not mention any topology for $\cD$. The right choice of topology is an important point which we will discuss in Section~\ref{sec:topology}.

Finally, for $\bfx\in\cD$ we denote by
\begin{equation}\label{eq:notation-cBMs}
\bfY^\bfx_t=\{Y_t^\ssup{x} \,|\, x\in\bfx\},\qquad t\ge0
 \end{equation}
a system of \emph{coalescing} Brownian motions starting from $\bfx$.
Here, $Y_t^\ssup{x}$ is the position at time $t$ of the particle in the system which started in $x\in\bfx$ at time $t=0$.
We may imagine that each particle follows the paths of a Brownian motion starting in $x$ until it collides with another motion, upon which the two particles are `merged' and evolve together.

}

\section{{Results}}\label{sec:classification}
In this section we state our main results. 
We will classify entrance laws for aBMs by embedding $\cD$ into a compact space and extending $\bfX$ to a Feller process on this space, for which all entrance laws are closable, and which we can describe explicitly.

\subsection{Entrance laws}
Recall that a family $\mu=(\mu_t)_{t>0}$ of probability measures on (the Borel $\sigma$-algebra of) $\cD$ is called a \emph{probability entrance law} for the semigroup $(P_t)_{t\ge0}$ if 
\begin{equation}\label{eq:entrance_law}\mu_sP_{t-s}=\mu_t \qquad \text{for all }0<s<t.\end{equation} 
See e.g. \cite[ Appendix A.5]{Li} or \cite{Sharpe} for the general theory of entrance laws.
Roughly speaking, an entrance law corresponds to a Markov process $(\bfX_t)_{t>0}$ with time-parameter set $(0,\infty)$ and `without initial condition', whose one-dimensional distributions are given by~$\mu_t$.

Let
\[ \cM_1(\bR):=\{u(x)\,dx\,|\,u:\bR\to[0,1]\text{ measurable}\}\]
denote the space of all absolutely continuous measures on $\R$ with densities taking values in $[0,1]$. We define an equivalence relation $\sim$ on $\cM_1(\R)$ by identifying $u$ with $1-u$
and consider the quotient space
\[ \cV := \cM_1(\R) /\! \sim.\]
We write $v=[u]=\{u,1-u\}$ for elements of $\cV$, i.e.\ for the equivalence classes under $\sim$.

Our main result, Theorem~\ref{thm:characterization_entrance_laws}, states that there is a bijective correspondence between probability entrance laws $(\mu_t)_{t>0}$ for the semigroup $(P_t)_{t\ge0}$ of aBMs on $\cD$
and probability measures $\nu$ on $\cV$. 
The subtle point is that this only works with 
the right topology on $\cD$, which we will describe in the next subsection.

\subsection{The topology on $\cD$}\label{sec:topology}

In order to turn $\cD$ into a topological space, as in \cite{TZ11} one may identify $\bfx\in\cD$ with the locally finite point measure $\sum_{x\in\bfx}\delta_x$,
thus embedding $\cD$ into the space of locally finite measures, and use the topology of vague convergence. 
Note however that employing this topology leads to c\`adl\`ag but not continuous paths for the process $\bfX$, since at annihilation events the total mass of the finite point measure changes.

We will introduce a different (weaker) topology on $\cD$ under which the paths of $\bfX$ are automatically continuous and which allows us to classify the entrance laws. 
The main idea is to regard the positions of the annihilating particles as `interfaces' of two measures on the real line with complementary support, and to use these measures to obtain a topology better adapted to the evolution of aBMs.
In order to make this precise, we need to introduce some additional notation and definitions. Recall that $\cM_1(\bR)$ denotes the space of all absolutely continuous measures on $\R$ with densities taking values in $[0,1]$.
We will usually use the same symbol to denote the absolutely continuous measure and its density.
We endow $\calM_1(\R)$ with the vague topology, i.e. $u^\ssup{n}\to u$ in $\cM_1(\bR)$ iff $\langle u^\ssup{n},\phi\rangle\to\langle u,\phi\rangle$ for all $\phi\in\cC_c(\bR)$.
It is easy to see that with this topology, $\cM_1(\R)$ is a compact space, see Lemma \ref{lem:compact} below. 
For $u\in\cM_1(\R)$, we define the \emph{interface} (of $u$ with its complement $1-u$) as
\[\cI(u):=\supp(u)\cap\supp(1-u),\]
where $\supp(u)$ denotes the measure-theoretic support of $u$, i.e.
\[ \supp(u):=\{x\in\R \,|\,  u\left(B_\eps(x)\right)>0 \text{ for all }\eps>0\} . \]
We call the elements of $\cI(u)$ \emph{interface points}.
{Note that $\cI(u)$ is always closed.}
The subspace of all $u\in\cM_1(\R)$ with discrete interface is denoted by
\[\cM_1^d(\R):=\{u\in\cM_1(\R)\,|\,\cI(u)\in\cD\},\]
which is dense in $\cM_1(\R)$, see Lemma \ref{lem:dense} below.
Note that for each $u\in\cM_1^d(\R)$, we may choose a version of its density taking values in $\{0,1\}$ and which is locally constant on each of the countably many disjoint open intervals in $\R\setminus\cI(u)$, where it takes the value $0$ or $1$ alternatingly.
In particular, for $u\in\cM_1^d(\R)$ the measure-theoretic and function-theoretic supports coincide. 

When restricted to $\cM_1^d(\R)$, the `interface operator' gives us a mapping $\cI:\cM_1^d(\R)\to\cD$ which is clearly surjective but not injective, since both $u$ and $1-u$ have the same interface.
Thus with the equivalence relation $\sim$ on $\cM_1(\R)$ identifying $u$ and $1-u$, we consider the quotient spaces 
\[ \cV^d := \cM_1^d(\R) /\! \sim\]
and $\cV = \cM_1(\R) /\!\!\sim$ introduced above. 
Endowed with the quotient topology, $\cV$ is also compact and $\cV^d$ is dense in $\cV$.
Note that the `interface operator' $\cI$ is well-defined on the equivalence classes and thus induces a mapping (which we denote by the same symbol)
\[\cI:\cV^d\to\cD,\]
which is easily seen to be a bijection and
induces in a canonical way a topology on $\cD$, generated by the system
\[\{\cI(U): U\subseteq\cV^d\text{ open}\}.\]
By definition, this is the coarsest topology on $\cD$ with respect to which $\cI^{-1}:\cD\to\cV^d$ is continuous, 
and with this topology $\cD$ is homeomorphic to $\cV^d$. We note that this topology on $\cD$ is strictly weaker than the topology used in \cite{TZ11}.

\subsection{{Classification of entrance laws}}

Now we return to the aBM process $(\bfX_t)_{t\ge0}$ on $\cD$ with semigroup $(P_t)_{t\ge0}$.
Via the homeomorphism $\cI^{-1}$, it induces a semigroup $(T_t)_{t\ge0}$ on $\cV^d$:
\begin{equation}\label{eq:defn_Q}
T_t(v;\cdot):=P_t(\cI(v);\cdot)\circ\cI,\qquad v\in\cV^d,\;t\ge0.
\end{equation}
Our main result states that this semigroup can be extended to a Feller semigroup $(\hat T_t)_{t\ge0}$ on the compact space $\cV$ 
which can be used to characterize the entrance laws for aBMs.
{In order to state this characterization precisely, we need the following notation: 

Given $u\in\cM_1(\R)$ and a system of cBMs $(\bfY_t^\bfx)_{t\ge0}$ 
starting from some $\bfx=(x_1,\ldots,x_n)\in\R^n$, let 
$\left(\chi_y\right)_{y\in\bfY_t^\bfx}$
be a family of Bernoulli random variables indexed by the cBM positions 
at time $t>0$ 
with conditional distribution 
\begin{equation}\label{defn:Bernoulli}
\cL\left(\big(\chi_y\big)_{y\in\bfY_t^\bfx}\,\big|\,\bfY^\bfx\right)=\bigotimes_{y\in\bfY_t^\bfx}\mathrm{Ber}\left(u(y)\right).
\end{equation}
We suppress the dependence on $u$ in the notation for these random variables.\footnote{Of course, the density $u\in\cM_1(\R)$ is only defined up to a Lebesgue-null set, but all our results will be independent of the version of $u$ we choose.}
Recalling  
our notation \eqref{eq:notation-cBMs}, note that for $i\ne j$ the random variables $\chi_{Y_t^\ssup{x_i}}$ and $\chi_{Y_t^\ssup{x_j}}$ are either identical or independent, 
depending on whether or not the Brownian motions starting from $x_i$ and $x_j$ have coalesced up to time $t$.
}

\newpage
\begin{theorem}\label{thm:characterization_entrance_laws}
Let $\cD$ be endowed with the topology introduced in Section \ref{sec:topology}.
\begin{itemize}
\item[a)] The semigroup $(T_t)_{t\ge0}$ on $\cV^d$ defined in \eqref{eq:defn_Q} can be extended to a Feller semigroup $(\hat T_t)_{t\ge0}$ on $\cV$ such that
\begin{equation}\label{eq:concentration}
\text{for all }v\in\cV\text{ and }t>0:\; \hat T_t(v;\cdot)\text{ is concentrated on }\cV^d.
\end{equation}
The corresponding Feller process, which we denote by $(V_t)_{t\ge0}$, has continuous paths
{
and is characterized by the following `moment duality':
writing $V_t=[U_t]=\{U_t,1-U_t\}\in\cV$,
we have for each $v=[u]\in\cV$, $n\in\N$ and Lebesgue-almost all $\bfx=(x_1,\ldots,x_{2n})\in\R^{2n}$ 
\begin{align}\bal\label{eq:duality_3}
\p_v\left(\bigcap_{i=1}^n\{U_t(x_{2i-1})=U_t(x_{2i})\}\right)=\p\left(\bigcap_{i=1}^n\big\{\chi_{Y_t^\ssup{x_{2i-1}}}=\chi_{Y_t^\ssup{x_{2i}}}\big\}\right),\qquad t>0,
\eal\end{align}
where $(\bfY_t^{\bfx})_{t\ge0}$ 
is a system of cBMs starting from $\bfx$ and 
{the Bernoulli random variables $\chi_{Y_t^\ssup{x_i}}$ are} as in \eqref{defn:Bernoulli}.
}
\item[b)]
There is a bijective correspondence between probability entrance laws $\mu=(\mu_t)_{t>0}$ for the semigroup $(P_t)_{t\ge0}$ of aBMs on $\cD$ and probability measures $\nu$ on $\cV$, given by the formula
\begin{equation}\label{eq:correspondence}
\mu_t=\nu \hat T_t\circ\cI^{-1}=\cL\left(\cI(V_t)\,|\,\p_{\nu}\right),\qquad t>0.
\end{equation}
{
For $\nu=\delta_v$ with $v=[u]\in\cV$, the corresponding entrance law $\mu$ is characterized by
\begin{equation}\label{eq:correspondence_0}
\p_{\mu}\left(\bigcap_{i=1}^n\{|\bfX_t\cap[x_{2i-1},x_{2i}]|\text{ is even}\}\right)=\p\left(\bigcap_{i=1}^n\big\{\chi_{Y_t^\ssup{x_{2i-1}}} = \chi_{Y_t^\ssup{x_{2i}}}\big\}\right)
\end{equation}
for all $t>0$, $n\in\N$ and $\bfx=(x_1,\ldots,x_{2n})\in\R^{2n,\uparrow}$. 
}
\end{itemize}
\end{theorem}

{
\begin{remark}
Note that since replacing $u$ by $1-u$ is equivalent (in distribution) to replacing $\chi_{Y_t^\ssup{x_i}}$ by $1-\chi_{Y_t^\ssup{x_i}}$, the RHS of \eqref{eq:duality_3} and \eqref{eq:correspondence_0} depends indeed only on the equivalence class $v=[u]\in\cV$.
Also note that if $u$ is $\{0,1\}$-valued, 
we can choose $\chi_{Y_t^\ssup{x_i}}\equiv u(Y_t^\ssup{x_i})$ and so in this case \eqref{eq:correspondence_0} reads
\begin{equation}\label{eq:border}
\p_{\mu}\left(\bigcap_{i=1}^n\left\{|\bfX_t\cap[x_{2i-1},x_{2i}]|\text{ is even}\right\}\right)=\p\left(\bigcap_{i=1}^n\{u(Y_t^\ssup{x_{2i-1}}) = u(Y_t^\ssup{x_{2i}})\}\right).
\end{equation}
We observe that this formula can be interpreted as an analogue in continuous space of the so-called \emph{border equation} characterizing annihilating random walks on $\Z$, see e.g.\ \cite[Sec.\ 2]{BG80}. 
\end{remark}
}

Our next result clarifies the question raised in the introduction concerning different approximations of the real line by asymptotically dense subsets.
In particular, it shows that each entrance law for aBMs can be approximated by a sequence of (random) initial conditions in $\cD$.

\begin{theorem}\label{thm:convergence}
Let $\cD$ be endowed with the topology introduced in Section \ref{sec:topology}, and let $(V_t)_{t\ge0}$ denote the Feller process from Theorem \ref{thm:characterization_entrance_laws}.
Let $(\mu^\ssup{n})_{n\in\N}$ be a sequence of probability measures on $\cD$, and consider the corresponding sequence 
of aBM processes started according to the (random) initial condition $\mu^\ssup{n}$. 
Then $\cL\big((\bfX_t)_{t>0}\,\big|\,\p_{\mu^\ssup{n}}\big)$ converges weakly in $\cC_{(0,\infty)}(\cD)$ 
iff the sequence $(\mu^\ssup{n}\circ\cI)_{n\in\N}$ of probability measures on $\cV^d$ converges weakly to some probability measure $\nu$ on $\cV$, in which case
\begin{equation}\label{eq:convergence}
\lim_{n\to\infty}\cL\big((\bfX_t)_{t>0}\,\big|\,\p_{\mu^\ssup{n}}\big)=\cL\big((\cI(V_t))_{t>0}\,\big|\, \p_{\nu}\big)\quad\text{on }\cC_{(0,\infty)}(\cD).
\end{equation}
Moreover, for any entrance law $(\mu_t)_{t>0}$ for the semigroup $(P_t)_{t\ge0}$ of aBMs there exists a sequence $(\mu^\ssup{n})_{n\in\N}$ of probability measures on $\cD$ such that 
\[\mu_t=\lim_{n\to\infty}\mu^\ssup{n}P_t,\qquad t>0.\] 
\end{theorem}

\begin{example}
To illustrate Theorems \ref{thm:characterization_entrance_laws} and \ref{thm:convergence}, we give various examples showing the effect of different ways of approximating increasingly dense initial conditions for aBMs, see also Figure \ref{fig:simulations}. 
\begin{itemize}
 \item First, consider $\bfx_n=\frac1n\bZ$. Clearly $\cI^{-1}(\bfx_n)$ converges to $[\frac12]$ in $\cV$, and hence by Theorem~\ref{thm:convergence} the system of aBMs starting from $\bfx_n$ converges. 
We have the same limit when $\bfx_n$ is any other regularly spaced lattice with mesh going to zero as $n\to\infty$, or when $\bfx_n$ is the realisation of a Poisson point process of intensity $n$, {as in Figure \ref{fig:simulations}(b)}. These approximations give the `maximal' entrance law considered in \cite{TZ11}. 
 \item
 In the example $\bfx_n=\frac1n\bZ+\{0,\frac1{n^2}\}$ we still have convergence of $\cI^{-1}(\bfx_n)$ in $\cV$, but the limit is $[0]$, which is degenerate and corresponds to the empty system. 
 So indeed in the limit the close-by pairs have annihilated and there are no surviving aBMs. 
 More generally, we have the same limit when $\bfx_n=\frac1n\bZ+\{0,\frac1{n^\alpha}\}$ for some $\alpha>1$.
\item We can also consider $\bfx_n=\frac1n\bZ+\{0,\frac1{4n}\}$, where $\cI^{-1}(\bfx_n)$ converges to $[\frac14]$ in $\cV$, which is different from $[\frac12]$. 
This is an example of an entrance law where we still start aBMs everywhere on the real line just as in $[\frac12]$, but the system `comes down from infinity' in a different way, giving rise to a different law of the aBMs. 
\begin{figure}[htp]
    \centering
    \begin{subfigure}[b]{0.48\textwidth}
\begin{tikzpicture}
\begin{axis}[enlargelimits=false, axis on top, axis equal image,ticks=none,axis lines = middle, 
x label style={at={(axis description cs:1.01,0.0)},anchor=west},
y label style={xshift=-0.5cm},
ylabel={t}, xlabel=$\bT$, 
width=1.1\textwidth]
\addplot graphics [xmin=0,xmax=100,ymin=0,ymax=70] {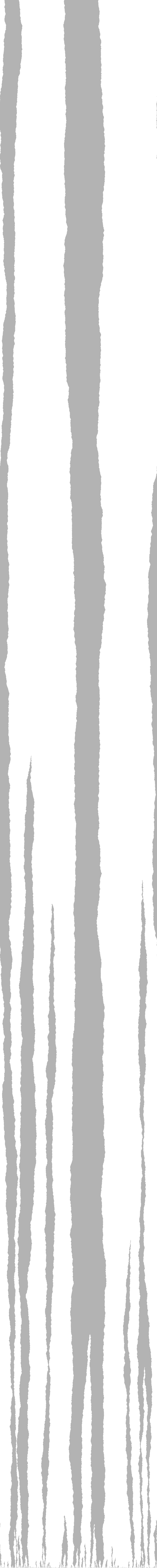};
\end{axis}
\end{tikzpicture}
        \caption{$\bfx_n^{(1)}=\frac1n\bZ$}
    \end{subfigure}
    \hfill
    \begin{subfigure}[b]{0.48\textwidth}
\begin{tikzpicture}
\begin{axis}[enlargelimits=false, axis on top, axis equal image,ticks=none,axis lines = middle, 
x label style={at={(axis description cs:1.01,0.0)},anchor=west},
y label style={xshift=-0.5cm},
ylabel={t}, xlabel=$\bT$, 
width=1.1\textwidth]
\addplot graphics [xmin=0,xmax=100,ymin=0,ymax=70] {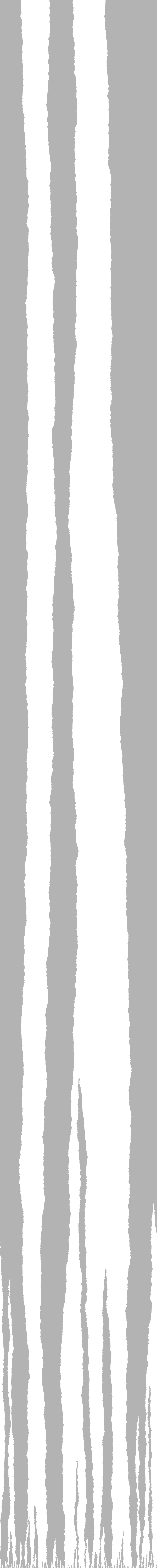};
\end{axis}
\end{tikzpicture}
       \caption{$\bfx_n^{(2)}\sim PPP(n)$}
    \end{subfigure}
\vspace{3mm}
\newline
    \begin{subfigure}[b]{0.48\textwidth}
\begin{tikzpicture}
\begin{axis}[enlargelimits=false, axis on top, axis equal image,ticks=none,axis lines = middle, 
x label style={at={(axis description cs:1.01,0.0)},anchor=west},
y label style={xshift=-0.5cm},
ylabel={t}, xlabel=$\bT$, 
width=1.1\textwidth]
\addplot graphics [xmin=0,xmax=100,ymin=0,ymax=70] {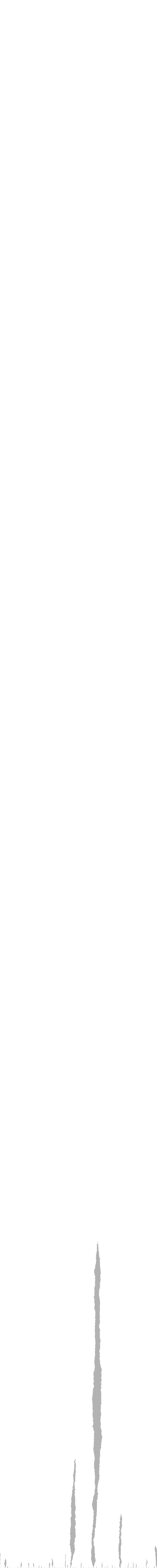};
\end{axis}
\end{tikzpicture}
      \caption{ $\bfx_n^{(3)}=\frac1n\bZ+\{0,\frac1{n^2}\}$}
    \end{subfigure}
		\hfill
		\begin{subfigure}[b]{0.48\textwidth}
\begin{tikzpicture}
\begin{axis}[enlargelimits=false, axis on top, axis equal image,ticks=none,axis lines = middle, 
x label style={at={(axis description cs:1.01,0.0)},anchor=west},
y label style={xshift=-0.5cm},
ylabel={t}, xlabel=$\bT$, 
width=1.1\textwidth]
\addplot graphics [xmin=0,xmax=100,ymin=0,ymax=70] {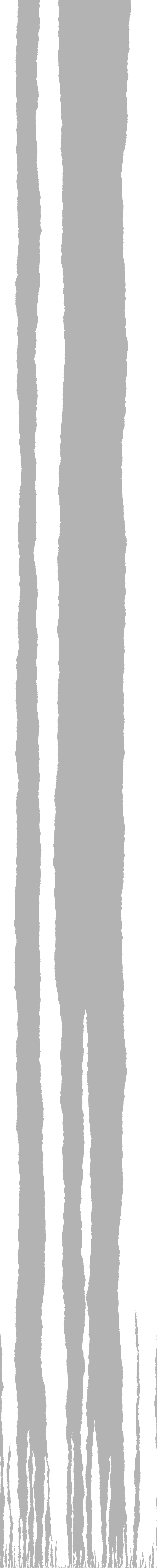};
\end{axis}
\end{tikzpicture}
	       \caption{ $\bfx_n^{(4)}=\frac1n\bZ+\{0,\frac1{4n}\}$}
    \end{subfigure}
 \caption{Simulations of aBMs {as interface process} with discrete starting configurations {on a torus $\bT$}.}
 \label{fig:simulations}
\end{figure}

\item 
As a final example we look at a sequence $\bfx_n\in\cD$ such that $\cI^{-1}(\bfx_n)$ does not converge in $\cV$: Let $x_n\in\R$ be a sequence converging monotone from below to some $a\in \bR$. We put $\bfx_n=\{x_1,\ldots,x_n\}$ and write $\cI^{-1}(\bfx_n)=[ u_n]$ with $u_n\in\cM_1^d(\R)$. Going from $\bfx_n$ to $\bfx_{n+1}$ adds the single point $x_{n+1}$, 
and we can choose the support of $u_{n+1}$ to remain fixed to the left of $x_{n+1}$, but such that it changes to the right of $x_{n+1}$. 
Then for any test function $\phi$ which is supported both to the left and to the right of $a$, the sequence $\langle u_n,\phi\rangle$ does not converge. 
However, if we add points in pairs, then the support of the induced measure remains unchanged except for the interval between the two added points, whose length goes to 0. 
Hence $\cI^{-1}(\bfx_{2n})$ and $\cI^{-1}(\bfx_{2n+1})$ converge to two distinct limit points. This is not surprising, since aBMs are parity preserving, and if we start with an even number eventually all will annihilate, while if we start with an odd number there will be a single surviving Brownian motion. 
However, if we extend the example to two sequences $x_n \uparrow a$ and $y_n \downarrow b$, $a<b$ and $\bfx_n=\{x_1,y_1,\ldots,x_n,y_n\}$, then the number of starting points is always even, but still $\cI^{-1}(\bfx_n)$ does not converge in $\cV$.
Note that we needed here that the sequence $x_n$ converges to a finite point $a\in \bR$. If $a=\infty$, the above argument does not work and in fact $\cI^{-1}(\bfx_n)$ does converge in $\cV$.
\end{itemize}
\end{example}

{
\subsection{Results on $n$-point densities}\label{sec:densities}
}
In this subsection, we turn to the $n$-particle density function for aBMs, which is defined as follows: 
If $\mu=(\mu_t)_{t>0}$ is an entrance law for the semigroup $(P_t)_{t\ge0}$, the corresponding $n$-point density is given by
\begin{align}\label{eq:density} 
p_{\mu}(t,\bfx)&:=
\lim_{\epsilon\to0}\frac{1}{(2\epsilon)^n}\,\bP_{\mu}\left(\bigcap_{i=1}^n\{\bfX_t\cap[x_i-\epsilon,x_i+\epsilon]\ne\emptyset\}\right),
\end{align}
for $\bfx=(x_1,\ldots,x_n)\in \bR^{n,\uparrow},\;t>0$.
See e.g.\ \cite[Appendix B]{MRTZ2006} for the existence of this density.

For $n=1$, the $1$-point density can be computed explicitly:
\begin{theorem}\label{thm:density}
Let $v=[u]\in\cV$ and consider the entrance law corresponding to $\nu:=\delta_{[u]}$ in view of Thm.\ \ref{thm:characterization_entrance_laws}. 
Then the $1$-particle density function is given by
\begin{align}\label{eq:1-point-density}
p_{[u]}(t,x) = \frac{1}{2\pi t^{2}}\, \int_{\bR^2}u(x+y_1)(1-u(x+y_2))|y_2-y_1|e^{-\frac{|\bfy|^2}{2t}}d\bfy,\qquad x\in \bR,\;t>0.
\end{align}
\end{theorem}

\begin{rem}
Observe that for the class of \emph{homogeneous} entrance laws parametrized by $[\lambda]\in\cV$, $\lambda\in[0,\frac12]$, the expression \eqref{eq:1-point-density} for the one-point density simplifies to
\[ p_{[\lambda]}(t,x) = \frac{2\lambda(1-\lambda)}{\sqrt{\pi t}}. \]

In particular, for the `maximal' entrance law for aBMs considered in \cite{TZ11}, corresponding to $\lambda=\frac{1}{2}$, we have
\[p_{[\frac12]}(t,x) = \frac{1}{2\sqrt{\pi t}}, \]
which indeed clearly maximizes the one-point density among all homogeneous entrance laws. 
Note that (as is to be expected from the thinning relation) this is half the density under the maximal entrance law for cBMs, compare \cite[Prop.\ 2.7]{SSS17}. 

However, non-homogeneous entrance laws can achieve bigger densities. For example, 
{
if we choose $u:=\ind_{\R^-}$, then the entrance law 
$\delta_{[\ind_{\R^-}]}$
}
corresponds to a single 
Brownian motion starting at the origin, for which we have
\[ p_{[\ind_{\R^-}]}(t,x)=\frac{1}{\sqrt{2\pi t}}e^{-\frac{x^2}{2t}}, \]
and in particular 
\[ p_{[\ind_{\R^-}]}(t,0)= \frac{1}{\sqrt{2\pi t}}>\frac{1}{2\sqrt{\pi t}}=p_{[\frac12]}(t,0). \]
This phenomenon is not limited to entrance laws which do not start densely: For $\epsilon\in(0,\frac12)$, consider the entrance law { corresponding to $\delta_{[u]}$ with
\[ u:=\epsilon+(1-2\epsilon)\ind_{\R^-}. \]
}
Here $\cI([u])=\bR$, but $u(x)\to\ind_{\R^-}$ uniformly as $\epsilon\to0$, and by \eqref{eq:1-point-density} $p_{\bullet}(t,x)$ is continuous in the uniform topology, so that $p_{[u]}(t,0)>p_{[\frac12]}(t,0)$ for $\epsilon$ small enough. 
We conclude that a more appropriate name for
the entrance law corresponding to $[\frac12]\in\cV$ and discussed in \cite{TZ11} would be `maximal homogeneous'.
\end{rem}

Turning to the case $n\ge2$, we have the following result:

{
\begin{prop}\label{prop:density-n}
Let $v=[u]\in\cV$ and consider the entrance law corresponding to $\nu:=\delta_{[u]}$ in view of Thm.\ \ref{thm:characterization_entrance_laws}. 
Then for each $n\in\N$, $\bfx=(x_1,\ldots,x_n)\in\R^{n,\uparrow}$ and $t>0$, we have 
\begin{align}\label{eq:n-point-density}
p_{[u]}(t,\bfx) 
&= \lim_{\eps\downarrow0} \frac{1}{(2\eps)^{n}}\,\p\bigg(\bigcap_{i=1}^n \{\chi_{Y^\ssup{x_i-\epsilon}_t}\ne \chi_{Y^\ssup{x_i+\epsilon}_t} \}\bigg),
\end{align}
where 
$(\bfY_t^{(x_1-\epsilon,\,x_1+\epsilon,\ldots,x_n-\epsilon,\,x_n+\epsilon)})_{t\ge0}$ is a system of cBMs 
and 
the Bernoulli random variables $\chi_{Y_t^\ssup{x_i\pm\epsilon}}$ are as in \eqref{defn:Bernoulli}.
\end{prop}
}

{
\begin{rem}\label{rem:not_explicit} 
\begin{itemize}
\item[a)]
For $n=1$, the event on the RHS of \eqref{eq:n-point-density} simplifies to 
\[\{\tau^\ssup{x,\epsilon}>t\}\cap\{\chi_{Y_t^\ssup{x-\epsilon}}\ne\chi_{Y_t^\ssup{x+\epsilon}}\},\]
where $\tau^\ssup{x,\epsilon}$ is the coalescence time of the two Brownian motions. 
The $1$-point density \eqref{eq:1-point-density} can then be obtained by using the distribution of two Brownian motions conditioned not to collide up to time $t$, see the proof of Thm.\ \ref{thm:density}.
 \item[b)]
Note that for a homogeneous entrance law $v=[\lambda]$ with $\lambda\in(0,1)$ constant, the RHS of \eqref{eq:n-point-density} can be simplified: 
For $\epsilon$ small enough, successive terms in the intersection are either independent or share a Bernoulli random variable via coalesced Brownian motions. 
Partitioning the index set $\{1,\ldots,n\}$ into $K$ blocks $\{i_k,\ldots,i_{k+1}-1\}$ (with $i_1:=1$ and $i_{K+1}:=n+1$) based on the coalescence structure, so that different blocks are independent, \eqref{eq:n-point-density} simplifies to
\begin{align}\label{eq:n-point-density-2}\bal
p_{[\lambda]}(t,\bfx) &= \lim_{\eps\downarrow0} \frac{1}{(2\eps)^{n}}\,\E\left[\ind_{D_t^\ssup\epsilon}\,\prod_{k=1}^K (\lambda(1-\lambda))^{\lfloor\tfrac{1+i_{k+1}-i_{k}}2\rfloor}(1+\ind_{\{i_{k+1}-i_{k} \text{ is odd}\}})\right]\\
&=
\lim_{\eps\downarrow0} \frac{1}{(2\eps)^{n}}\,\E\left[\ind_{D_t^\ssup\epsilon}\, (\lambda(1-\lambda))^{\tfrac{K_{odd}+n}2}2^{K_{odd}}\right],
\eal\end{align}
where $D_t^\ssup\epsilon:=\bigcap_{i=1}^n \{ Y_t^\ssup{x_i-\epsilon} \neq Y_t^\ssup{x_i+\epsilon} \}$ and $K_{odd}:=\sum_{k=1}^K \ind_{\{i_{k+1}-i_{k} \text{ is odd}\}}$.
In particular, for the `maximal' entrance law $\lambda=\frac{1}{2}$ we have
\begin{align}
p_{[\frac12]}(t,\bfx)=&\lim_{\eps\downarrow0} \frac{1}{(4\eps)^{n}}\,\bP\left(D_t^\ssup\epsilon\right).
\end{align}
We see again that \eqref{eq:n-point-density-2} becomes maximal for $\lambda=\frac{1}{2}$, thus the $n$-point density function is maximized by $\lambda=\frac{1}{2}$ in the class of homogeneous entrance laws, for any $n\in\N$.
\end{itemize}
\end{rem}

The expression \eqref{eq:n-point-density} for the $n$-point density function is deceivingly short.
In fact, as we have just seen in Remark \ref{rem:not_explicit},
it 
involves a lot of combinatorial effort to disentangle the effect of the various ways the Brownian motions can have coalesced.
}
However, we can give a more tractable representation of a `thinned' version of the $n$-point density as follows:
Fix $\bfx\in\cD$. The discrete closed set $\bfx$ can be partitioned into two disjoint subsets $\bfx=\bfx^\ssup{1}\cup\bfx^\ssup{2}$ so that points in $\bfx$ are alternating between $\bfx^\ssup{1}$ and $\bfx^\ssup{2}$ and such that either $\sup(\bfx) \in \bfx^\ssup{1}$ or otherwise $\inf(\bfx\cap [0,\infty)) = \inf(\bfx^\ssup{1} \cap[0,\infty))$. Denote by $\bfx^{thin}$ the random subset of $\bfx$ equalling either $\bfx^\ssup{1}$ or $\bfx^\ssup{2}$ with probability $\frac12$. 

Now for an entrance law $\mu=(\mu_t)_{t>0}$, we define the thinned $n$-point density as
\begin{align}
p_\mu^{thin}(t,\bfx):= \lim_{\epsilon\to0}\frac{1}{(2\epsilon)^n}\,\bP_{\mu}\left(\bigcap_{i=1}^n\{\bfX^{thin}_t\cap[x_i-\epsilon,x_i+\epsilon]\ne\emptyset\}\right),\qquad \bfx\in \bR^{n,\uparrow},\;t>0,
\end{align}
where $\bfX_t^{thin}$ is the random subset of $\bfX_t$ obtained via thinning as defined above.

{
\begin{prop}\label{thm:n-point-density-thinned}
Let $v=[u]\in\cV$ and consider the entrance law corresponding to $\nu:=\delta_{[u]}$ in view of Thm.\ \ref{thm:characterization_entrance_laws}. 
Then 
for $\bfx\in\bR^{n,\uparrow}$ and $t>0$, we have
\begin{align}\label{eq:n-density}
&p_{[u]}^{thin}(t,\bfx) \\
&= \frac{q(t,\bfx)}{2} \,\lim_{\epsilon\downarrow0}\E\bigg[\prod_{k=1}^n u( B^\ssup{x_k-\epsilon}_t)(1-u(B_t^\ssup{x_k+\epsilon})) + \prod_{k=1}^n (1-u(B_t^\ssup{x_k-\epsilon}))u(B^\ssup{x_k+\epsilon}_t) \,\bigg|\, \tau^\ssup{\bfx,\epsilon}>t\bigg],
\end{align}
where $B$ is a Brownian motion in $\R^{2n}$ starting from and indexed by $(x_1-\epsilon,x_1+\epsilon,\ldots,x_n-\epsilon,x_n+\epsilon)$, 
\[\tau^\ssup{\bfx,\epsilon}:=\inf\{t>0\,|\,B_t^\ssup{y}=B_t^\ssup{z}\text{ for some }y\neq z\}\]
is the first collision time of any pair of coordinates, and
\begin{align}\label{eq:q_t}
q(t,\bfx):=\lim_{\epsilon\to0}\frac{1}{(2\epsilon)^n}\,\bP_{}(\tau^\ssup{\bfx,\epsilon}>t).
\end{align}
\end{prop}
}

\begin{rem} We recall that~\cite{TZ11} show that the
$n$-point densities for aBMs started in the `maximal homogeneous' entrance law  are given in terms of Pfaffians. 
It would be interesting to make the connection to our formulae, which however 
does not seem to be completely straight-forward, see also Remark~\ref{rem:not_explicit}.
\end{rem}

\section{The continuous-space  voter model}\label{sec:cSSM}

The proof of Theorem \ref{thm:characterization_entrance_laws} (the characterization of entrance laws for annihilating Brownian motions) in Section \ref{sec:proofs} below relies on a close connection of aBMs to what we call the \emph{continuous-space voter model}.
This section is devoted to a survey explaining this connection, which is also of independent interest. We will not give proofs but refer to the existing literature, commenting on necessary modifications when appropriate.

{
We start by recalling the classical (nearest-neighbor) voter model on $\bZ^d$: This is a Markov process $(\eta_t)_{t\ge0}$ taking values in $\{0,1\}^{\bZ^d}$ 
such that 
 \begin{equation}\label{eq:voter-discrete}
 \eta(x)\text{ flips to }1-\eta(x) \text{ at rate }\frac{1}{2d}\sum_{y:|y-x|=1}\ind_{\{\eta(y)\ne \eta(x)\}},\qquad x\in\Z^d.
 \end{equation}
As is well known, 
it is characterized by the following \emph{moment duality:} 
For all $\eta\in\{0,1\}^{\bZ^d}$ and finite subsets $A\subset\bZ^d$, we have
\begin{equation}\label{eq:moment-duality-discrete}
\bE_{\eta}\Bigg[\prod_{x\in A}\eta_t(x)\Bigg]=\bE_{A}\Bigg[\prod_{x\in A_t} \eta(x)\Bigg],\qquad t\ge0,
\end{equation}
where $(A_t)_{t\ge0}$ denotes a (set-valued) system of (instantaneously) coalescing nearest-neighbor random walks starting from $A$.
See \cite{Liggett85} for this duality and for further background on the discrete voter model.

In view of this, a continuous-space analogue of the voter model should be a Markov process $(u_t)_{t\ge0}$ with $u_t(x)\in\{0,1\}$ for all $x\in\R$ which is characterized by a similar moment duality, but where the coalescing random walks are replaced by coalescing Brownian motions. 
Indeed, such a model was first introduced by \cite{Evans1997} in a much more general context and further discussed in \cite{Donnellyetal2000} and \cite{Zhou2003}, 
where it is however called a \emph{continuum-sites stepping-stone model}. 
The following result is essentially contained as a special case (only two types, Brownian migration on $\R$) in \cite[Thm.\ 4.1, Prop.\ 5.1]{Evans1997} and \cite[Cor.\ 7.3]{Donnellyetal2000}:
\begin{theorem}[\cite{Evans1997, Donnellyetal2000}]
\label{thm:voter} There exists a unique Feller semigroup
$(Q_t)_{t\ge0}$ on $\cM_1(\R)$ such that the corresponding Feller process $(u_t)_{t\ge0}$ is characterized by the following moment duality: 
For all $u\in\cM_1(\R)$ and $n\in\N$, 
we have for Lebesgue-almost all $\bfx=(x_1,\ldots,x_n)\in\R^n$
\begin{equation}\label{eq:duality_pointwise}
\E_{u}\bigg[\prod_{i=1}^nu_t(x_i)\bigg]
 =\E\bigg[\prod_{y\in\bfY^\bfx_t}u(y)\bigg],
\qquad t\ge0,
\end{equation}
where $\bfY^\bfx$ 
denotes a system of coalescing Brownian motions starting from $\bfx$. 
For each initial condition $u\in\cM_1(\R)$, the process $(u_t)_{t\ge0}$ has continuous sample paths, and for each fixed $t>0$ we have, almost surely under $\p_u$, 
\begin{equation}\label{eq:sep_types}
 u_t(x)\in\{0,1\}\qquad\text{for Lebesgue-almost all }
 x\in\bR.
 \end{equation}
\end{theorem}

In view of the `separation of types'-property \eqref{eq:sep_types} as well as the analogous form of the moment dualities \eqref{eq:moment-duality-discrete} and \eqref{eq:duality_pointwise}, we prefer 
to call the Feller process $(u_t)_{t\ge0}$ from Theorem \ref{thm:voter} the \emph{continuous-space voter model}, and will denote it by $\mathrm{CSVM}$ in the following. 
If we refer to a particular initial condition $u\in\cM_1(\R)$, we write $\mathrm{CSVM}_u$.
Note that \eqref{eq:duality_pointwise} implies that the model is symmetric under exchange of $u$ and $1-u$, in the sense that
 \begin{equation}\label{eq:symmetric}
 \cL\left((1-u_t)_{t\ge0})\,|\,\p_{u}\right)=\cL\left((u_t)_{t\ge0})\,|\,\p_{1-u}\right),\qquad u\in\cM_1(\R).
 \end{equation}
The qualification 
`Lebesgue-almost all' in Theorem \ref{thm:voter} is necessary 
since the process $(u_t)_{t\ge0}$ is \emph{measure-valued} (recall that on the state space $\cM_1(\R)$, we use the topology of vague convergence). 
However, we will see below that a version of the densities $u_t$ can be chosen such that the moment duality \eqref{eq:duality_pointwise} holds 
for all $\bfx\in\R^n$, even in a pathwise sense and not only in expectation (see \eqref{eq:duality_pathwise}).
In particular, 
$u_t(x)$ is a Bernoulli random variable with parameter $\E_x[u(B_t)]$ for each fixed $t>0$ and $x\in\R$, where $B$ is a standard Brownian motion.
Moreover, we will see that \eqref{eq:sep_types} can be strengthened to a much stronger clustering property, see Thm.\ \ref{thm:interfaces}.

There are several possible constructions for CSVM. In \cite{Evans1997}, Evans 
constructed the model directly from (a weak form of) the moment duality \eqref{eq:duality_pointwise}, even for much more general particle motions than Brownian motion on $\R$, 
and for an \emph{uncountable} type space instead of the two-type case we consider here. 
In several later papers, CSVM was shown to arise as the limit of various other discrete- or continuous-space models. 
It was first observed in \cite[p.\ 794]{AS11} (although without a formal proof) that the discrete one-dimensional voter model converges to $\mathrm{CSVM}$ under diffusive space/time-rescaling.
Later, CSVM was obtained as the scaling limit of rescaled spatial $\Lambda$-Fleming-Viot processes on $\R$ (see \cite[Thm.\ 1.1]{BEV13}), 
of rescaled interacting Fleming-Viot processes on $\Z$
under diffusive space/time-rescaling (see \cite[Thms.\ 1.31, 1.32]{GSW16}), 
and of continuous-space \emph{stepping stone models}
\begin{align}\label{eq:SSM}
 \tfrac{\partial}{\partial t} u^\ssup{\gamma}_t(x) = \tfrac{1}{2}\Delta u^\ssup{\gamma}_t(x)  + \sqrt{\gamma u^\ssup{\gamma}_t(x)(1-u^\ssup{\gamma}_t(x)) }\, \dot{W}_t(x),\qquad x\in\R
\end{align}
as $\gamma\to\infty$ (see \cite[Thm.\ 2.8a)]{HOV15}). 
The latter result makes sense in view of the fact that as first observed in \cite{Shiga86}, 
the stepping stone model \eqref{eq:SSM} satisfies an analogous moment duality as in \eqref{eq:duality_pointwise}, but where the dual is a system 
of \emph{delayed} cBMs, where
two Brownian motions coalesce when their intersection local time exceeds an independent exponential random variable with parameter $\gamma$,
and which clearly converges to a system of instantaneously coalescing Brownian motions as $\gamma\to\infty$.

The dynamics of the discrete voter model is specified explicitly by its flip rates \eqref{eq:voter-discrete}. In contrast, in \cite{Evans1997} the 
CSVM-process $(u_t)_{t\ge0}$ was specified only indirectly via the moment duality \eqref{eq:duality_pointwise}.
In analogy with the graphical representation of the discrete model (see e.g.\ \cite[Sec.\ III.6]{Liggett85}), it is possible to give a more explicit \emph{graphical} or \emph{genealogical} construction of CSVM,
as has been first observed in \cite{AS11} and then carried out rigorously in \cite{GSW16}. This however requires use of the (dual) \emph{Brownian web}, a highly non-trivial object. 
We will briefly explain this graphical construction at the end of the present section.

An alternative `explicit' construction of CSVM, which also provides the link to annihilating Brownian motions, is possible in terms of \emph{interfaces}.
Indeed, for the one-dimensional discrete model, 
it is quite easy to see (e.g.\ by the graphical representation) 
and was first observed in \cite{S78} that the dynamics of the discrete interface
\[I_t:=\{x\in\Z\,|\, \eta_t(x)\ne \eta_t(x+1)\}\]
follows an \emph{annihilating random walk}, and this provides an equivalent description of the voter model on $\Z$.
It was shown in \cite{HOV15} that the analogous assertion holds in continuous space also. 
In other words, the interface model of CSVM is given by aBMs.
But in continuous space, this is much more subtle since there we may start from initial conditions whose interface does not consist of well-separated points.
In fact, even for an arbitrary initial condition $u\in\cM_1(\R)$, the process `locally comes down from infinity' immediately in the sense that almost surely, the set $\cI(u_t)$ is discrete for each $t>0$, and the movement of the interface points for positive times is described in law by a system of aBMs. 
}

A mathematically precise formulation is as follows: For $u\in\cM_1^d(\R)$, consider a system of aBMs $(\bfX_t)_{t\ge0}$ started from the discrete closed set $\cI(u)\in\cD$.
The system $(\bfX_t)_{t\ge0}$ induces a (random) partition of $[0,\infty)\times\bR$, whose components are bounded by the closure 
\[\cJ := {\rm cl}  \{ (t,\bfX_t)\, | \, t \in [0,\infty) \} \]
of the graphs of the annihilating paths. 
The path properties of the annihilating system ensure that the components of this partition can be assigned the value $0$ or $1$ in an alternating fashion, i.e.\ so that neighboring components have always different values.
That is, we define a random mapping 
\begin{equation}\label{eq:hat m}\hat m:\cJ^c\to\{0,1\}\end{equation}
with the property that it is locally constant and alternating on (each component of) $\cJ^c$.
See Figure~\ref{fig:colouring-1} for an illustration. 

 \begin{figure}
\begin{center}
 \includegraphics{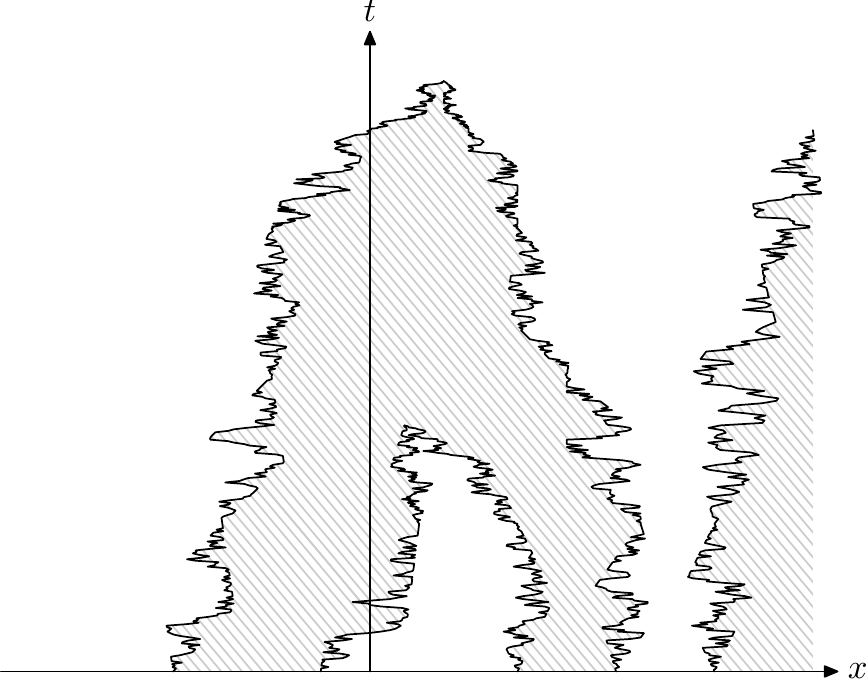}
 \caption{
 An illustration of the construction of the process $(\hat U_t)_{t\ge0}\in\calC_{[0,\infty)}(\cM^d_1(\R))$ from an initial configuration $u$ with five interfaces and a system of aBMs starting from $\cI(u)$.
 The density $\hat U_t(x)$ is equal to $1$ for all $(t,x)$ in the shaded area and is zero otherwise.}
  \label{fig:colouring-1}
\end{center} \end{figure}

Of course, given the aBM path $\bfX$ there are exactly two possibilities to choose the mapping $\hat m$ on $\cJ^c$. 
However, it is determined by the choice of $\hat m(0,\cdot)$ at time zero.
This enables us to define an $\cM_1^d(\R)$-valued process $(\hat U_t)_{t\ge0}$ as follows: Given 
$u\in \cM_1^d(\R)$, 
let
\[\hat m(0,x):=\ind_{\supp(u)}(x),\qquad x\in\R,\] 
then extend $\hat m(0, \cdot)$ to $[0,\infty)\times\R$ by the requirement that it is constant on the components of $\cJ^c$ as described above, and set 
\begin{equation}\label{eq:defn_U_hat}\hat U_t(x):=\ind_{\hat m(t,x)=1},\qquad x\in\R,\; t\geq 0;\end{equation}
again see Figure~\ref{fig:colouring-1}. 
By the definition of the topology on $\cM_1^d(\R)$, it is clear that the process $(\hat U_t)_{t\ge0}$
is a random element of $\cC_{[0,\infty)}(\cM_1^d(\bR))$.

Then the following theorem is contained as a special case 
in \cite[Thms. 2.12, 2.14]{HOV15}:

\begin{theorem}[\cite{HOV15}]\label{thm:interfaces}
Let $(u_t)_{t\ge0}$ be the $\mathrm{CSVM}_u$-process from Theorem \ref{thm:voter}.
\begin{itemize}
\item[a)] 
Suppose $u\in\cM_1^d(\R)$, i.e.\ $\cI(u)\in \cD$. 
If we let the process $(\hat U_t)_{t\ge0}$ be defined as in \eqref{eq:defn_U_hat} above, then
\begin{equation*}
(u_t)_{t\ge0}\overset{d}=(\hat U_t)_{t\ge0}\quad\text{on }\cC_{[0,\infty)}(\cM_1(\bR)).
\end{equation*}
\item[b)] Let $u\in\cM_1(\R)$. 
Then, almost surely, we have $\cI(u_{t})\in \cD$ for all $t>0$. 
Moreover, for any $t_0>0$, the evolution of $(u_t)_{t\ge t_0}$ is given (in law) as in a) when started in $u_{t_0}$.
\end{itemize}\end{theorem}

We note that \cite[Thm.\ 10.2]{Donnellyetal2000} contains a somewhat analogous result for the corresponding continuous-space voter model with Brownian migration on the \emph{torus}. 
(In that case, by compactness, the system comes down to \emph{finitely} many interfaces immediately.)
We also remark that \cite{Zhou2003} studies clustering behavior for the model with Brownian migration on the real line as in our case, but with infinitely many types as in \cite{Evans1997}. 
In particular, \cite[Thm.\ 3.7]{Zhou2003} (when restricted to the two-types case) shows essentially that under homogeneous initial conditions $u_0\equiv u\in(0,1)$, for each \emph{fixed} $t>0$, the interface $\cI(u_t)$ is discrete almost surely.
This is not strong enough to give rise to an `interface process' as in Thm.\ \ref{thm:interfaces}.

\begin{rem}
\begin{enumerate}
\item Theorem \ref{thm:interfaces} provides the crucial link between annihilating Brownian motions and the continuous-space voter model. 
In particular, we will see in the proof that for the process $(V_t)_{t \geq 0}$ from Theorem~\ref{thm:characterization_entrance_laws}, if $V_0 = v = [u]$, then we have $(V_t)_{t\ge0} \stackrel{d}{=} ([u_t])_{t\ge0}$.
\item 
Write $(Q_t)_{t\ge0}$ for the semigroup of $(u_t)_{t\ge0}$ as in Thm.\ \ref{thm:voter}. Then Theorem \ref{thm:interfaces}b) implies in particular that 
\begin{equation}\label{eq:clustering} 
\text{for all }u\in\cM_1(\R)\text{ and }t>0:\; Q_t(u;\cdot)\text{ is concentrated on }\cM_1^d(\R). 
\end{equation}
\end{enumerate}
\end{rem}

{We conclude this section with an outline of the graphical construction of CSVM in terms of the (dual) Brownian web, first conceived in \cite{AS11} and later elaborated in \cite{GSW16}.
For the analogous graphical representation of the discrete voter model, see e.g.\ \cite[Sec.\ III.6]{Liggett85}.
Our exposition will be non-technical; for background and technical details concerning the Brownian web, in particular the precise state space and topology involved, we refer to \cite{SSS17}.
Actually, we will use the 
\emph{double Brownian web} $(\cW,\widehat\cW)$, which is a coupled construction of the Brownian web and its dual on the same probability space.
That is, $\cW=\{\cW^\ssup{t,x}\}_{(t,x)\in\R^2}$ resp. $\widehat\cW=\{\widehat\cW^\ssup{t,x}\}_{(t,x)\in\R^2}$ is a (random) collection of paths 
indexed by the starting position $(t,x)$ and evolving forwards resp. backwards in time as coalescing Brownian motions,
and such that almost surely, no path in $\cW$ crosses any path in $\widehat\cW$. 
More precisely, for each $(t,x)\in\R^2$, almost surely there is a unique path $\cW^\ssup{t,x}$ starting from $(t,x)$ and not crossing any path in $\widehat\cW$, 
and for any finite collection $(t_1,x_1),\ldots,(t_n,x_n)$ of starting points, $(\cW^\ssup{t_1,x_1},\ldots, \cW^\ssup{t_n,x_n})$ is distributed as a system of $n$ coalescing Brownian motions,
and analogously for $\widehat\cW$.

Now the graphical construction of CSVM works as follows: 
Suppose first that $u\in\cM_1^d(\R)$, so that we may assume $u(x)\in\{0,1\}$ for all $x\in\R$.
Then we set
\begin{equation}
u_t(x):=u(\widehat\cW_0^\ssup{t,x})\in\{0,1\},\qquad t\ge0,\; x\in\R.
\end{equation}
This gives a coupled definition of the Bernoulli random variables $u_t(x)$ for all $(t,x)$ and has an interpretation in terms of \emph{genealogies}: 
namely, the type of an `individual' at time $t>0$ and site $x\in\R$ is determined by tracing back its genealogy along the backwards coalescing path $\widehat \cW^\ssup{t,x}$ starting from $(t,x)$ down until time zero. 
(Note the analogy with the graphical representation of the discrete voter model.)
For a general initial condition $u\in\cM_1(\R)$, in order to define $u_t(x)$, one traces back the genealogy in the same way and then in the last step, conditionally on the realization of $\widehat\cW$, one samples a type in $\{0,1\}$ according to a Bernoulli distribution with parameter $u(\widehat\cW_0^\ssup{t,x})$.
Some care is needed to ensure that the resulting construction of $u_t(\cdot)$ is measurable in $x$.\footnote{In fact, one is tempted to just take a family $(\chi_x)_{x\in\R}$ of independent Bernoulli random variables with parameter $u(x)$, respectively, and to put $u_t(x):=\chi_{\widehat\cW_0^\ssup{t,x}}$.
This however does not work since $x\mapsto \chi_x$ is not measurable.}
Observing that the set
\[\widehat E:=\left\{\widehat\cW_0^\ssup{t,x}\,|\,t>0,\,x\in\R\right\}\]
is countable almost surely (see \cite[p.\ 15]{GSW16}),
we let $(U_n)_{n\in\N}$ be a sequence of i.i.d.\ uniform random variables on $[0,1]$, independent of $\widehat\cW$, and given a realization of $\widehat\cW$, we let $J:\widehat E\to\N$ be an enumeration of the elements of $\widehat E$. 
Then we put for $t>0$ and $x\in\R$
\begin{equation}\label{eq:graphical-construction}
 u_t(x):=\chi_{\widehat\cW_0^\ssup{t,x}}:=\begin{cases}1&\text{if }U_{J(\widehat\cW_0^\ssup{t,x})}\le u(\widehat\cW_0^\ssup{t,x}),\\0&\text{else}.\end{cases}
\end{equation}

It is easy to check that conditionally on $\widehat\cW$, the random variables 
$\chi_{\widehat\cW_0^\ssup{t,x}}$ (as a family indexed by the elements of $\widehat E$)
are independent and Bernoulli distributed with parameter $u(\widehat\cW_0^\ssup{t,x})$, 
and that $u_t(\cdot)$ is measurable in $x$, hence yields a random element of $\cM_1(\R)$.

In order to check that the process $(u_t)_{t\ge0}$ defined above agrees with the continuous-space voter model, note first that \eqref{eq:graphical-construction} implies the moment duality \eqref{eq:duality_pointwise} in a pathwise sense, i.e. 
\begin{equation}\label{eq:duality_pathwise}
 \prod_{i=1}^n u_t(x_i)=\prod_{i=1}^n \chi_{\widehat\cW_0^\ssup{t,x_i}}, \qquad t>0,\;\bfx=(x_1,\ldots,x_n)\in\R^n.
 \end{equation}
(In the terminology of \cite{JK14}, we have a strong pathwise duality.)
Taking expectations, we obtain immediately the moment duality \eqref{eq:duality_pointwise}, since 
$\{\widehat\cW_0^\ssup{t,x_1},\ldots,\widehat\cW_0^\ssup{t,x_n}\}\overset{d}{=}\bfY^\bfx_t$ for a system $(\bfY^\bfx)_{t\ge0}$ of (forward) cBMs starting from $\bfx$. 
Because the moments determine the distribution, this shows that the one-dimensional marginals of $(u_t)_{t\ge0}$ defined in \eqref{eq:graphical-construction} agree with those of CSVM.
In order to conclude that the processes agree in distribution on path space, one needs to check that $(u_t)_{t\ge0}$ is a Markov process and has continuous paths; 
for a proof of this in a more general (infinitely-many-types) context, see \cite[Thm.\ 1.27]{GSW16}.

This pathwise `graphical' construction of CSVM allows one to quickly see the assertions of Thm.\ \ref{thm:interfaces}. 
For example, to see that the system locally comes down from infinity immediately, we fix $t>0$, $x\in\R$ and define
\[J_t(x):=\{y\in\R \,|\, \widehat\cW_0^\ssup{t,y} = \widehat\cW_0^\ssup{t,x}\}.\]
The 
properties of $\widehat\cW$ imply that $J_t(x)$ is a non-degenerate interval 
for each $x\in\R$, and by \eqref{eq:graphical-construction} we have $u_t(y)=u_t(x)$ for all $y\in J_t(x)$.
Thus considered as a measure, $u_t$ is 
supported on a countable collection of disjoint intervals. Moreover, the boundaries of these intervals cannot accumulate, since otherwise there would exist a finite interval containing at time $0$ infinitely many endpoints of backward coalescing paths starting at time $t>0$, 
contradicting the fact that these cBMs form a locally finite system at time $0$. In other words, we must have $u_t\in\cM_1^d(\R)$, which is Thm.\ \ref{thm:interfaces}b).
Note that 
so far we used only the dual (`backward') Brownian web $\widehat\cW$. 
However, by the coupling with the `forward' Brownian web $\cW$ it is quite easy to see the assertion of Thm.\ \ref{thm:interfaces}a) that starting from $u\in\cM_1^d(\R)$, interface points move as annihilating Brownian motions.
For example, suppose that $u$ is supported on countably many disjoint compact intervals, 
i.e.\ 
$u=\sum_{k\in\Z}\ind_{[a_k,b_k]}$ with $a_k<b_k<a_{k+1}$ for all $k\in\Z$. 
Then by \eqref{eq:graphical-construction} we have $u_t(x)
=\sum_{k=1}^n \ind_{J_t^\ssup{k}}(x)$, 
where $J_t^\ssup{k}:=\left\{x\in\R\,|\,\widehat\cW_0^\ssup{t,x}\in[a_k,b_k]\right\}$. 
By the properties of $\widehat\cW$, it is clear that each $J_t^\ssup{k}$ is a finite non-degenerate interval $J_t^\ssup{k}=[a_k(t),b_k(t)]$ provided it is not empty. 
We claim that for those indices $k$ such that $J_t^\ssup{k}\ne\emptyset$, we have
\[a_k(t)=\cW_t^\ssup{0,a_k}\qquad\text{and}\qquad b_k(t)=\cW_t^\ssup{0,b_k}.\]
Indeed, this follows from the fact that no path in $\cW$ can cross any path in $\widehat\cW$. 
This shows that (locally) interface points move as Brownian motions as long as $J_t^\ssup{k}\ne\emptyset$.
The coalescence time $\tau$ of the forward paths $\cW^\ssup{0,a_k}$ and $\cW^\ssup{0,b_k}$ is precisely the time from which on no backward path in $\widehat\cW$ can end up in the interval $[a_k,b_k]$ at time zero, and so $J_t^\ssup{k}=\emptyset$ for $t\ge\tau$,
which means that the interfaces corresponding to the $k$-th interval annihilate at time $\tau$. 
Obviously, analogous arguments work also for all other kinds of initial conditions in $\cM_1^d(\R)$, for example when there are only finitely many interfaces and/or when the support of $u$ is unbounded in one or both directions.

We thus see that the graphical construction of CSVM sketched above allows for simple and elegant arguments in situations where the proofs would be more involved if one had to use only the moment duality \eqref{eq:duality_pointwise}. 
On the other hand, it relies on properties of the (dual) Brownian web, a highly non-trivial object. 
We will use the pathwise duality \eqref{eq:duality_pathwise} in some of the proofs in Section \ref{sec:proofs} below, but emphasize that all results in this paper can also be proved without recourse to the Brownian web, using essentially only \eqref{eq:duality_pointwise}. 
Moreover, the moment duality technique is robust to some degree and can be generalized to situations where an analogous web construction does not (yet) exist. 
(Recall that \cite{Evans1997} allowed for much more general migration mechanisms than Brownian motion on $\R$, and \cite{HOV15} considered interface points which move as aBMs with a highly non-regular drift.) 
}

\section{Proofs of results
}\label{sec:proofs}
In this section, we prove our results stated above in Section~\ref{sec:classification}. 

\begin{proof}[Proof of Thm.\ \ref{thm:characterization_entrance_laws}]
For the proof of part a), recall that $(P_t)_{t\ge0}$ denotes the semigroup of annihilating Brownian motions on $\cD$, and that the semigroup $(T_t)_{t\ge0}$ on $\cV^d$ is defined as the image of $(P_t)_{t\ge0}$ under the inverse interface operator $\cI^{-1}:\cD\to\cV^d$, see \eqref{eq:defn_Q}. 
The latter can be rewritten as
\begin{equation}\label{identity_semigroups}
T_t(v;f\circ\cI)=P_t(\cI(v);f),\qquad v\in\cV^d,\;f\in\cC_b(\cD). 
\end{equation}

On the other hand, recall that $(Q_t)_{t\ge0}$ denotes the Feller semigroup on $\cM_1(\R)$ corresponding to the continuous-space voter process $(u_t)_{t\ge0}$ from Thm.\ \ref{thm:voter}.  
Via the canonical quotient mapping $q:\cM_1(\R)\to\cV$, $(Q_t)_{t\ge0}$ factorizes to a Feller semigroup
\begin{equation}\label{qe:defn_Q}
\hat T_t(v;\cdot):=Q_t(u;\cdot)\circ q^{-1},\qquad v=[u]\in\cV,\;t\ge0
\end{equation}
on the quotient space $\cV$, and the corresponding Feller process 
\[(V_t)_{t\ge0}:=([u_t])_{t\ge0}\]
inherits the path continuity from $(u_t)_{t\ge0}$.
Of course, for \eqref{qe:defn_Q} to make sense we need to check in particular that the definition does not depend on the choice  of the representative $u$ or $1-u$ of the equivalence class $[u]$. 
But this (as well as the Feller property) 
follows easily from the symmetry \eqref{eq:symmetric}.
The property \eqref{eq:concentration} of the semigroup $(\hat T_t)_{t\ge0}$ follows immediately from the corresponding clustering property \eqref{eq:clustering} of $(Q_t)_{t\ge0}$.

In order to see that $(\hat T_t)_{t\ge0}$ as defined in \eqref{qe:defn_Q} is indeed an extension of $(T_t)_{t\ge0}$ as defined in \eqref{eq:defn_Q}, 
observe that for any $v\in\cV^d$ we have by Theorem \ref{thm:interfaces}a) that
\begin{equation}\label{representation_aBM}
\cL\big((\bfX_t)_{t\ge0}\,\big|\,\p_{\cI(v)}\big)=\cL\big((\cI(V_t))_{t\ge0}\,\big|\, \p_{v}\big)\quad\text{on }\cC_{[0,\infty)}(\cD).
\end{equation}
But together with \eqref{identity_semigroups}, this implies 
\begin{align}\bal
 T_t(v;f)&=P_t\big(\cI(v);f\circ\cI^{-1}\big)=\E_{\cI(v)}\left[f\circ\cI^{-1}(\bfX_t)\right]=\E_v[f(V_t)]
 =\hat T_t(v;f)
\eal\end{align}
for each $v\in\cV^d$, $f\in\cC_b(\cV^d)$ and $t\ge0$.

{
The `moment duality' formula \eqref{eq:duality_3} follows directly from the graphical construction \eqref{eq:graphical-construction} of the CSVM-process $(u_t)_{t\ge0}$,
since $(\widehat\cW_0^\ssup{t,x_{1}},\ldots,\widehat\cW_0^\ssup{t,x_{2n}})\overset{d}=(Y_t^\ssup{x_1},\ldots,Y_t^\ssup{x_{2n}})$ for a system $(\bfY^\bfx)_{t\ge0}$ of (forward) cBMs starting from $\bfx$.
It remains to show that this duality characterizes the law of $V_t$ on $\cV$ for fixed $t>0$. To this end, we argue as follows:
For $u\in\cM_1(\R)$ and $x,y\in\R$, define
\begin{align}\label{eq:definition-h}
h_u(x,y)\equiv h_{1-u}(x,y):=u(x)(1-u(y))+(1-u(x))u(y). 
\end{align}
Note that $h_u(\cdot,\cdot)$ depends only on the equivalence class $[u]\in\cV$, and that for $u\in\cM_1^d(\R)$ we may assume $u(\cdot)\in\{0,1\}$ and thus
\begin{align}\label{eq:equality-h}
h_u(x,y)=\ind_{\{u(x)\ne u(y)\}}\qquad\text{for } u\in\cM_1^d(\R)\text{ and Lebesgue-almost all }x,y\in\R.
\end{align}

Now given $n\in\N$ and $\phi_1,\ldots,\phi_{2n}\in\cC_c(\R)$, 
we define $\Phi:=\phi_1\otimes\cdots\otimes\phi_{2n}\in\cC_c(\R^{2n})$ and a function $F_{\Phi}:\cV\to\R$ by
\begin{align*}
F_{\Phi}(v)&:=\int_{\R^{2n}}\Phi(\bfx)\prod_{i=1}^n h_u(x_{2i-1},x_{2i})\,d\bfx\\
&=\prod_{i=1}^n\big(\langle u,\phi_{2i-1}\rangle\langle 1-u,\phi_{2i}\rangle+\langle u,\phi_{2i}\rangle\langle 1-u,\phi_{2i-1}\rangle\big),\qquad v=[u]\in\cV. 
\end{align*}
Note that $F$ is well-defined and continuous on $\cV$, and the family of functions
\[\cF:=\left\{F_{\Phi}(\cdot)\,|\,n\in\N,\phi_i\in\cC(\R)\text{ for } i=1,\ldots,2n\right\}\subseteq\cC(\cV)\]
is closed under multiplication and separates the points of $\cV$. 
Therefore 
$\cF$ is separating for probability laws on $\cV$. 
Since $V_t\in\cV^d$ for $t>0$ by \eqref{eq:concentration}, using \eqref{eq:equality-h} we obtain
\begin{align}
\E_v\left[F_{\Phi}(V_t)\right]
&=\int_{\R^{2n}}\Phi(\bfx)\,\p_v\left(\bigcap_{i=1}^n\big\{U_t(x_{2i-1})\ne U_t(x_{2i})\big\}\right) d\bfx,
\end{align}
showing that \eqref{eq:duality_3} determines the law of $V_t$ on $\cV$.
}
Thus part a) of Thm.\ \ref{thm:characterization_entrance_laws} is proved.

For part b), let $\nu$ be any probability measure on $\cV$ and define $\mu_t$ by \eqref{eq:correspondence}. 
Then by \eqref{identity_semigroups} we have for any $0<s<t$ and $f\in\cC_b(\cD)$ that
\[\bal\mu_t(f)&=\nu\hat T_t(f\circ \cI)=\int_{\cV^d}\hat T_{t-s}(\cdot;f\circ\cI)\, d(\nu\hat T_s)\\
&=\int_{\cV^d}P_{t-s}(\cI(\cdot);f)\,d(\nu\hat T_s)=\int_\cD P_{t-s}(\cdot,f)\,d\mu_s,\eal\]
showing that $\mu_t=\mu_sP_{t-s}$ and $(\mu_t)_{t>0}$ is an entrance law for the semigroup $(P_t)_{t\ge0}$ of aBMs.
Conversely, suppose that $(\mu_t)_{t>0}$ is any such entrance law. 
Then a similar calculation shows that $\nu_t:=\mu_t\circ \cI$ defines an entrance law for the semigroup $(\hat T_t)_{t\ge0}$ on $\cV$. 
Since the latter is a Feller semigroup on a compact space, the entrance law $(\nu_t)_{t>0}$ is closable, i.e.\ there exists a probability measure $\nu$ on $\cV$ such that $\nu_t=\nu\hat T_t$ for all $t>0$. 
In fact, let $\cP(\cV)$ denote the space of all probability measures on $\cV$ endowed with the topology of weak convergence. 
Then $\cP(\cV)$ is itself compact, and thus there exists a sequence $t_n\downarrow0$ and $\nu\in\cP(\cV)$ such that $\nu_{t_n}\to\nu$ weakly as $n\to\infty$. 
Then for any $f\in\cC_b(\cV)$, $t>0$ and $n\in\N$ large enough we have by the Feller property that
\[\nu_t(f)=\nu_{t_n}\hat T_{t-t_n}(f)\to \nu\hat T_t(f),\]
i.e.\ $\nu_t=\nu\hat T_t$, and we conclude that $\mu_t=\nu_t\circ\cI^{-1}=\nu\hat T_t\circ\cI^{-1}$. 
Moreover, the moment duality~\eqref{eq:duality_3} implies that different probability measures $\nu\ne\tilde\nu$ on $\cV$ lead to different laws
\begin{equation}\label{eq:unique}\cL\big(V_t\,|\,\p_{\nu}\big)\ne \cL\big(V_t\,|\,\p_{\tilde\nu}\big)\qquad \text{for }t>0. \end{equation}
Thus the mapping defined by \eqref{eq:correspondence} is indeed a bijection, and 
{
the claimed correspondence 
}
is established.
{
For $\nu=\delta_v$ with $v=[u]\in\cV$ and the corresponding entrance law $\mu$, \eqref{eq:correspondence} implies
\begin{align*}
\p_{\mu}\left(\bigcap_{i=1}^n\{|\bfX_t\cap[x_{2i-1},x_{2i}]|\text{ is even}\}\right)&=\p_{v}\left(\bigcap_{i=1}^n\{|\cI(V_t)\cap[x_{2i-1},x_{2i}]|\text{ is even}\}\right)
\end{align*}
for each $\bfx\in\R^{2n,\uparrow}$ and $t>0$. 
Again writing $V_t=[U_t]=\{U_t,1-U_t\}$, we
observe that on $\{\cI(V_t)\cap\bfx=\emptyset\}$ (an event of full probability), we have that
$|\cI(V_t)\cap[x_{2i-1},x_{2i}]|$ is even iff $U_t(x_{2i-1})=U_t(x_{2i})$, for all $i=1,\ldots,n$. 
Thus by \eqref{eq:duality_3}, the previous display equals
\begin{align*}
\p_v\left(\bigcap_{i=1}^n\{U_t(x_{2i-1}) = U_t(x_{2i})\}\right)=\p\left(\bigcap_{i=1}^n\big\{\chi_{Y_t^\ssup{x_{2i-1}}} = \chi_{Y_t^\ssup{x_{2i}}}\big\}\right),
\end{align*}
establishing formula \eqref{eq:correspondence_0}. 
}
\end{proof}

\begin{proof}[Proof of Thm.\ \ref{thm:convergence}]
Suppose that $(\mu^\ssup{n})_{n\in\N}$ is a sequence of probability measures on $\cD$ such that $\nu^\ssup{n}:=\mu^\ssup{n}\circ\cI$ converges weakly to some probability measure $\nu$ on $\cV$. 
Then by the Feller property of $(V_t)_{t\ge0}$, we have $\cL\big((V_t)_{t\ge0}\,\big|\, \p_{\nu^\ssup{n}}\big)\to\cL\big((V_t)_{t\ge0}\,\big|\, \p_{\nu}\big)$ weakly on $\cC_{[0,\infty)}(\cV)$, 
thus also 
\begin{equation}\label{eq:proof_convergence}
\cL\big((V_t)_{t>0}\,\big|\, \p_{\nu^\ssup{n}}\big)\to\cL\big((V_t)_{t>0}\,\big|\, \p_{\nu}\big) \quad\text{weakly on }\cC_{(0,\infty)}(\cV^d).\end{equation}
By the continuous mapping theorem and \eqref{representation_aBM}, it follows that
\[\cL\big((\bfX_t)_{t>0}\,\big|\,\p_{\mu^\ssup{n}}\big)=\cL\big((\cI(V_t))_{t>0}\,\big|\,\p_{\nu^\ssup{n}}\big)\to\cL\big((\cI(V_t))_{t>0}\,\big|\, \p_{\nu}\big)\quad\text{on }\cC_{(0,\infty)}(\cD),\]
i.e.\ \eqref{eq:convergence}. 
Conversely, suppose that $\cL\big((\bfX_t)_{t>0}\,\big|\,\p_{\mu^\ssup{n}}\big)=\cL\big((\cI(V_t))_{t>0}\,\big|\, \p_{\mu^\ssup{n}\circ\cI}\big)$ converges weakly in $\cC_{(0,\infty)}(\cD)$. 
Consider the sequence  $\nu^\ssup{n}:=\mu^\ssup{n}\circ\cI$ of probability measures on $\cV^d$. Since $\cV^d\subseteq\cV$ and $\cV$ is compact, this sequence is relatively compact w.r.t.\ the topology of weak convergence. 
Moreover, by continuous mapping also $\cL\big((V_t)_{t>0}\,\big|\,\p_{\nu^\ssup{n}} \big)$ converges weakly in $\cC_{(0,\infty)}(\cV)$. 
But this implies (see \eqref{eq:proof_convergence}) that for any two limit points $\nu$ and $\tilde\nu$ of the sequence  $(\nu^\ssup{n})_{n\in\N}$, we must have 
$\cL\big((V_t)_{t>0}\,\big|\, \p_{\nu}\big)=\cL\big((V_t)_{t>0}\,\big|\, \p_{\tilde\nu}\big)$ on $\cC_{(0,\infty)}(\cV)$ and thus $\nu=\tilde\nu$ (recall \eqref{eq:unique}).
Thus $(\nu^\ssup{n})_{n\in\N}$ must converge weakly to some probability measure $\nu$ on $\cV$. 

Finally, let $(\mu_t)_{t>0}$ be any probability entrance law for the semigroup $(P_t)_{t>0}$ of aBMs on $\cD$. Let $\nu$ be the (unique) probability measure on $\cV$ corresponding to the entrance law in view of Theorem \ref{thm:characterization_entrance_laws}.
Since $\cV^d$ is dense in $\cV$, there is a sequence $(\nu^\ssup{n})_{n\in\N}$ of probability measures concentrated on $\cV^d$ such that $\nu^\ssup{n}\to\nu$ weakly as $n\to\infty$. 
Then putting $\mu^\ssup{n}:=\nu^\ssup{n}\circ \cI^{-1}$ and again using the Feller property of $(\hat T_t)_{t\ge0}$ as well as \eqref{identity_semigroups}, we get
\[\mu_t=\nu\hat T_t\circ\cI^{-1}=\lim_{n\to\infty}\nu^\ssup{n}\hat T_t\circ\cI^{-1}=\lim_{n\to\infty}\nu^\ssup{n}T_t\circ\cI^{-1}=\lim_{n\to\infty}\mu^\ssup{n}P_t,\qquad t>0,\]
concluding the proof.
\end{proof}

We continue with the proofs of the results in Section~\ref{sec:densities}. 

\begin{proof}[Proof of Proposition \ref{prop:density-n}]
{
Let $v=[u]\in\cV$ and assume that $u$ is not identically $0$ or $1$, since otherwise the statement is trivial.
We will show that for each $\bfx\in \bR^{n,\uparrow}$ and $t>0$, we have
\begin{align}\label{eq:density-variant}
p_{[u]}(t,\bfx) = 
\lim_{\epsilon\downarrow0}\frac{1}{(2\epsilon)^n}\, \bP_v\left(\bigcap_{i=1}^n \big\{|\bfX_t\cap [x_i-\epsilon,x_i+\epsilon]|\text{ is odd}\big\}\right).
\end{align}

In order to show \eqref{eq:density-variant}, let $(u_t)_{t \geq 0}$ be the $\mathrm{CSVM}_u$-process from Theorem~\ref{thm:voter}.
By the representation of entrance laws in Theorem~\ref{thm:characterization_entrance_laws}b), the $n$-particle density is given by
\begin{equation}
p_{[u]}(t,\bfx)=\lim_{\epsilon\to0}\frac{1}{(2\epsilon)^n}\bP_{u}\left(\bigcap_{i=1}^n\big\{\cI(u_t)\cap[x_i-\epsilon,x_i+\epsilon]\neq\emptyset\big\}\right),\qquad \bfx\in \bR^{n,\uparrow},\;t>0. 
\end{equation}

First we argue that for any $x\in\bR$,
\begin{align}\label{eq:one-interface-only}
\lim_{\epsilon\to0} \bP_{u}\Big( |\cI(u_t)\cap [x-\epsilon,x+\epsilon]|>1 \;\Big|\;  |\cI(u_t)\cap [x-\epsilon,x+\epsilon]|\geq 1 \Big)=0.
\end{align}
We will use the graphical construction of $u_t(x)=\chi_{\widehat\cW_0^\ssup{t,x}}$ via the dual Brownian web, recall \eqref{eq:graphical-construction}. 
Again we may assume that $\cI(u_t)\cap\{x-\epsilon,x+\epsilon\}=\emptyset$ since this event has full probability under $\p_u$.
By the coalescence property, $\widehat\cW^\ssup{t,x-\epsilon}_0 = \widehat\cW^\ssup{t,x+\epsilon}_0$ implies that $\widehat\cW^\ssup{t,y_1}_0=\widehat\cW^\ssup{t,y_2}_0$ for all $y_1,y_2\in [x-\epsilon,x+\epsilon]$, and therefore $\cI(u_t)\cap [x-\epsilon,x+\epsilon]=\emptyset$. Therefore the existence of an interface point in $[x-\epsilon,x+\epsilon]$ implies $\widehat\cW^\ssup{t,x-\epsilon}_0\neq \widehat\cW^\ssup{t,x+\epsilon}_0$. 
Similarly, the existence of two interface points in $[x-\epsilon,x+\epsilon]$ implies that there exists $y\in[x-\epsilon,x+\epsilon]$ so that $\widehat\cW^\ssup{t,x-\epsilon}_0\neq \widehat\cW^\ssup{t,y}_0$ and $\cW^\ssup{t,y}_0\neq \cW^\ssup{t,x+\epsilon}_0$. 
As $\epsilon\to0$, the probability that three Brownian motions started in $[x-\epsilon,x+\epsilon]$ do not meet decays faster than the probability that $\widehat\cW^\ssup{t,x-\epsilon}_0\neq \widehat\cW^\ssup{t,x+\epsilon}_0$, proving 
\begin{align}\label{eq:one-interface2}
\lim_{\epsilon\to0} \bP_{}\left( |\cI(u_t)\cap [x-\epsilon,x+\epsilon]|>1 \;\middle|\;  \widehat\cW^\ssup{t,x-\epsilon}_0\neq \widehat\cW^\ssup{t,x+\epsilon}_0 \right)=0.
\end{align}
Further, note that $\{u_t(x-\epsilon)\neq u_t(x+\epsilon)\}$ 
implies the existence of an interface point in $[x-\epsilon,x+\epsilon]$, thus 
\begin{align*}
&\frac{\bP_{}\left( |\cI(u_t)\cap [x-\epsilon,x+\epsilon]|>1 \right)} {\bP_{}\left(|\cI(u_t)\cap [x-\epsilon,x+\epsilon]|\geq 1 \right)}	
\leq \frac{\bP_{}\left( |\cI(u_t)\cap [x-\epsilon,x+\epsilon]|>1 \right)} {\bP_{}\left( u_t(x-\epsilon)\neq u_t(x+\epsilon) \right)}	\\
&= \frac{\bP_{}\left( |\cI(u_t)\cap [x-\epsilon,x+\epsilon]|>1 \;\middle|\; \widehat\cW^\ssup{t,x-\epsilon}_0\neq \widehat\cW^\ssup{t,x+\epsilon}_0 \right)} {\bP_{}\left( u_t(x-\epsilon)\neq u_t(x+\epsilon) \;\middle|\; \widehat\cW^\ssup{t,x-\epsilon}_0\neq \widehat\cW^\ssup{t,x+\epsilon}_0 \right)}.
\end{align*}
Together with \eqref{eq:one-interface2} and the fact that
\[ \liminf_{\epsilon\to0} \bP\left(u_t(x-\epsilon)\neq u_t(x+\epsilon)\;\middle|\; \widehat\cW^\ssup{t,x-\epsilon}_0\neq \widehat\cW^\ssup{t,x+\epsilon}_0 \right)>0,\]
we obtain \eqref{eq:one-interface-only}.
But this clearly implies
\begin{align}
p_{[u]}(t,x)&=\lim_{\epsilon\to0}\frac{1}{2\epsilon}\,\bP_{u}\left(\big\{\cI(u_t)\cap[x-\epsilon,x+\epsilon]\neq\emptyset\big\}\right)\\
&=\lim_{\epsilon\to0}\frac{1}{2\epsilon}\,\bP_{u}\left(\big\{|\cI(u_t)\cap[x-\epsilon,x+\epsilon]|=1\big\}\right)
\end{align}
and therefore also 
\[p_{[u]}(t,x)=\lim_{\epsilon\to0}\frac{1}{2\epsilon}\,\bP_{u}\left(\big\{|\cI(u_t)\cap[x-\epsilon,x+\epsilon]|\text{ is odd}\big\}\right),\qquad x\in\R,\; t>0,\]
which is \eqref{eq:density-variant} for $n=1$. Clearly, at the expense of more notation, this argument can be generalized, proving \eqref{eq:density-variant} for arbitrary $n\in\N$.
Now \eqref{eq:n-point-density} follows directly from \eqref{eq:correspondence_0}.
}
\end{proof} 

\begin{proof}[Proof of Thm.\ \ref{thm:density}]
{
By Proposition \ref{prop:density-n}, we know that
\begin{align*}
p_{[u]}(t,x) = \lim_{\epsilon\to0} \frac{1}{2\epsilon}\, \p\left(\chi_{Y_t^\ssup{x-\epsilon}}\ne\chi_{Y_t^\ssup{x+\epsilon}}\right),\qquad x\in \bR,\;t>0. 
\end{align*}
With $\tau^\ssup{x,\epsilon}:=\inf\{s>0:Y_{s}^\ssup{x-\epsilon}=Y_{s}^\ssup{x+\epsilon}\}$ 
denoting the coalescence 
time of the two Brownian motions, we use that the above probability is zero conditioned on $\{\tau^\ssup{x,\epsilon}\leq t\}$. 
Together with the fact that conditionally on $\{\tau^\ssup{x,\epsilon}>t\}$, the random variables $\chi_{Y_t^\ssup{x\pm\epsilon}}$ are independent and Bernoulli distributed with parameter $u(Y_t^\ssup{x\pm\epsilon})$, we obtain that
\begin{align}
 \p\left(\chi_{Y_t^\ssup{x-\epsilon}}\ne\chi_{Y_t^\ssup{x+\epsilon}}\right)
 &=\E\left[\ind_{\{\tau^\ssup{x,\epsilon}>t\}}\,h_u(Y_t^\ssup{x-\epsilon},Y_t^\ssup{x+\epsilon})\right],
\end{align}
where 
$h_u(\cdot,\cdot)$ is the function defined in \eqref{eq:definition-h}. Consequently,
\begin{align}\label{eq:cond-not-meet}
p_{[u]}(t,x) = \lim_{\epsilon\to0} \frac{1}{2\epsilon}\, \bE\left[h_u(Y_t^\ssup{x-\epsilon},Y_t^\ssup{x+\epsilon})\;\middle|\; \tau^\ssup{x,\epsilon}> t\right]\bP(\tau^\ssup{x,\epsilon}> t).
\end{align}
The distribution of a two-dimensional Brownian motion ``started at $(0,0)$'' and conditioned to stay in $\bR^{2,\uparrow}$ for the time interval $[0,t]$ is known 
and has density
\begin{align}\label{eq:2-non-col-BM}
\frac{y_2-y_1}{2t^{3/2}\sqrt{\pi}}e^{-\frac{|\bfy|^2}{2t}},\qquad \bfy=(y_1,y_2) \in \bR^{2,\uparrow},
\end{align}
see e.g.\ \cite{KT03}, eq.\ (2.10).
Therefore, as $\epsilon\to0$, the conditional expectation in \eqref{eq:cond-not-meet} converges to 
\begin{align*}
\int_{\bR^{2,\uparrow}} h_u(x+y_1,x+y_2)\frac{y_2-y_1}{2t^{3/2}\sqrt{\pi}}e^{-\frac{|\bfy|^2}{2t}} d\bfy.
\end{align*}
Moreover, by the reflection principle and the fact that the difference $Y_t^\ssup{x+\epsilon}-Y_t^\ssup{x-\epsilon}$ is a Brownian motion running at twice the speed, we have
\[ \bP(\tau^\ssup{x,\epsilon}>t) = 1-2\bP_0(B_{2t}>2\epsilon) = \int_{-2\epsilon}^{2\epsilon} \frac{1}{\sqrt{2\pi(2t)}}e^{-\frac{r^2}{4t}}dr = \frac{2\epsilon}{\sqrt{\pi t}} + o(\epsilon). \]
Thus, we obtain from \eqref{eq:cond-not-meet} that
\begin{align*}
p_{[u]}(t,x) = \frac{1}{2\pi t^2}\int_{\bR^{2,\uparrow}} h_u(x+y_1,x+y_2)(y_2-y_1)e^{-\frac{|\bfy|^2}{2t}} d\bfy.
\end{align*}
By symmetry of $h_u$, we also have
\begin{align*}
p_{[u]}(t,x) = \frac{1}{2\pi t^2}\int_{\bR^{2,\downarrow}} h_u(x+y_1,x+y_2)|y_2-y_1|e^{-\frac{|\bfy|^2}{2t}} d\bfy,
\end{align*}
and hence
\begin{align*}
p_{[u]}(t,x) 
&= \frac{1}{4\pi t^2}\int_{\bR^{2}} h_u(x+y_1,x+y_2)|y_2-y_1|e^{-\frac{|\bfy|^2}{2t}} d\bfy	\\
&= \frac{1}{4\pi t^2}\int_{\bR^{2}} 2u(x+y_1)(1-u(x+y_2))|y_2-y_1|e^{-\frac{|\bfy|^2}{2t}} d\bfy.
\end{align*}
}
\end{proof}

\begin{proof}[Proof of Prop.\ \ref{thm:n-point-density-thinned}]
Fix $\bfx\in \bR^{n, \uparrow}$. Consider $u\in\cM^d_1(\bR)$.  
Then we partition $\cI(u)$ into two disjoint sets $\cI_1(u)$ and $\cI_2(u)$ by saying that any $  z \in \cI(u)$ is also in $\cI_1(u)$   if $u([z-\eps,z])(1-u)([z,z+\eps])) > 0$
for all sufficiently small $\eps > 0$ and setting $\cI_2(u) := \cI(u) \setminus \cI_1(u)$.
{
Then the thinning procedure corresponds to randomly choosing $\cI_1(u)$ or $\cI_2(u)$ with probability $\tfrac12$ each.
Similarly to the proof of Proposition \ref{prop:density-n} (recall \eqref{eq:one-interface-only}), we obtain that
\begin{align}
&\lim_{\epsilon\downarrow0} \frac{1}{(2\epsilon)^n}\,\bP_u\left(\bigcap_{i=1}^n \big\{\cI_1(u_t)\cap [x_i-\epsilon,x_i+\epsilon] \neq \emptyset\big\} \right)\\
&=\lim_{\epsilon\to0}\frac{1}{(2\epsilon)^n}\,\p_u\left(\bigcap_{i=1}^n\big\{|\cI_1(u_t)\cap[x_i-\epsilon,x_i+\epsilon]|=|\cI(u_t)\cap[x_i-\epsilon,x_i+\epsilon]|=1\big\}\right)\\
&=\lim_{\epsilon\to0}\frac{1}{(2\epsilon)^n}\,\p_u\left(\bigcap_{i=1}^n\big\{u_t(x_i-\epsilon)=1,u_t(x_i+\epsilon)=0\big\}\cap\big\{|\cI(u_t)\cap[x_i-\epsilon,x_i+\epsilon]|=1\big\}\right)\\
&=\lim_{\epsilon\to0}\frac{1}{(2\epsilon)^n}\,\p_u\left(\bigcap_{i=1}^n\big\{u_t(x_i-\epsilon)=1,u_t(x_i+\epsilon)=0\big\}\right),
\end{align}
where for the second equality we used that on the event $\big\{|\cI(u_t)\cap[x_i-\epsilon,x_i+\epsilon]|=1\big\}$ that there is exactly one interface point in $[x_i-\epsilon,x_i+\epsilon]$, this interface point belongs to $\cI_1(u_t)$ iff $u_t(x_i-\epsilon)=1$ and $u_t(x_i+\epsilon)=0$. 
Now we use again the graphical construction \eqref{eq:graphical-construction} and argue as in the proof of Thm.\ \ref{thm:density}: 
With $\tau^\ssup{\bfx, \epsilon}$ denoting the first collision time in the system of coalescing Brownian motions $\widehat\cW_0^\ssup{t,x_i\pm\epsilon}$, 
we have
\begin{align}
 &\p_u\left(\bigcap_{i=1}^n\big\{u_t(x_i-\epsilon)=1,u_t(x_i+\epsilon)=0\big\}\right)=\p\left(\bigcap_{i=1}^n\big\{\chi_{\widehat\cW_0^\ssup{t,x_i-\epsilon}}=1,\chi_{\widehat\cW_0^\ssup{t,x_i+\epsilon}}=0\big\}\right)\\
&=\p\left(\bigcap_{i=1}^n\big\{\chi_{\widehat\cW_0^\ssup{t,x_i-\epsilon}}=1,\chi_{\widehat\cW_0^\ssup{t,x_i+\epsilon}}=0\big\}\,\middle|\,\{\tau^\ssup{\bfx,\epsilon}>t\}\right)\,\p(\tau^\ssup{\bfx,\epsilon}>t)\\
&=\E\left[\prod_{i=1}^n u(\widehat\cW_0^\ssup{t,x_i-\epsilon})(1-u(\widehat\cW_0^\ssup{t,x_i+\epsilon}))\,\middle|\, \tau^\ssup{\bfx,\epsilon}>t\right]\,\p(\tau^\ssup{\bfx,\epsilon}>t).
\end{align}
and thus
\begin{align*}
&\lim_{\epsilon\downarrow0} \frac{1}{(2\epsilon)^n}\,\bP_u\left(\bigcap_{i=1}^n \big\{\cI_1(u_t)\cap [x_i-\epsilon,x_i+\epsilon] \neq \emptyset\big\} \right)\\
&= \lim_{\epsilon\downarrow0} \frac{1}{(2\epsilon)^n}\,\E\left[\prod_{i=1}^n u(\widehat\cW_0^\ssup{t,x_i-\epsilon})(1-u(\widehat\cW_0^\ssup{t,x_i+\epsilon}))\,\middle|\, \tau^\ssup{\bfx,\epsilon}>t\right]\,\p(\tau^\ssup{\bfx,\epsilon}>t).
\end{align*}

Analogously,
\begin{align*}
&\lim_{\epsilon\downarrow0} \frac{1}{(2\epsilon)^n}\,\bP_u\left(\bigcap_{i=1}^n \big\{\cI_2(u_t)\cap [x_i-\epsilon,x_i+\epsilon] \neq \emptyset\big\} \right)\\
&= \lim_{\epsilon\downarrow0} \frac{1}{(2\epsilon)^n}\,\bE\left[\prod_{i=1}^n (1-u(\widehat\cW^\ssup{t,x_i-\epsilon}_0))u(\widehat\cW^\ssup{t,x_i+\epsilon}_0)\,\middle|\,\tau^\ssup{\bfx,\epsilon}>t\right]\,\p(\tau^\ssup{\bfx,\epsilon}>t).
\end{align*}
With $q(t,\bfx)=\lim_{\epsilon\to0}\frac{1}{(2\epsilon)^n}\,\p(\tau^\ssup{\bfx,\epsilon}>t)$, the claim follows by choosing $\cI_1$ or $\cI_2$ with probability $\tfrac12$.
}
\end{proof}

\appendix
\section{Appendix}

\subsection{On the construction of annihilating and coalescing Brownian motions}\label{sec:construction}
The construction of a \emph{finite} system of aBMs (or cBMs) is of course straightforward: If $\bfx\subseteq\R$ is finite, take a collection of independent Brownian motions $\{(B_t^\ssup
{x})_{t\ge0}:x\in\bfx\}$ indexed by and starting from $\bfx$, and let them run until the first collision time
\[\tau:=\inf\{t>0\,|\, \exists\, y,z\in\bfx,\,y\ne z: B_t^\ssup{y}= B_t^\ssup{z}\}>0.\]
Note that the collision pair $(y,z)$ is uniquely defined. At time $\tau$, restart the system with the new initial condition where the collision pair is removed from the configuration. 
An analogous procedure works for cBMs: 
{
Here, we do not remove the collision pair, but replace it by (two copies of) a single Brownian motion, reflecting the fact that the colliding particles `merge' and evolve together.
}

If the initial condition $\bfx\in\cD$ is infinite, clearly the above procedure does not work any longer since it may happen that the first collision time $\tau$ equals zero with positive probability. 
The obvious idea to deal with this problem is to approximate the (discrete, hence countable) set $\bfx=\{x_1,x_2,\ldots\}$ by finite sets $\bfx_{n}:=\{x_1,\ldots,x_n\}$ 
and to show that as $n\to\infty$, the system of aBMs $\bfX^{\bfx_n}$ 
converges in a suitable sense to $\bfX^\bfx$. 
See for example \cite[Sec. 4.1]{TZ11} for a weak convergence approach which works for both aBMs and cBMs. 

Another possibility is to define the infinite annihilating system by restriction from the corresponding infinite coalescing system, the latter of which can be constructed by \emph{monotonicity}: 
For each $n\in\N$, let $(\bfY^{\bfx_n}_t)_{t\ge0}$ denote a system of cBMs starting from the finite set $\bfx_n$.
It is well-known that there exists a coupling such that almost surely,
\[\bfY^{\bfx_n}_t\subseteq\bfY^{\bfx_{n+1}}_t\qquad\text{for all }t>0,\,n\in\bN.\]
Using this, the infinite coalescing system $(\bfY_t^{\bfx})_{t\ge0}$ can be constructed pathwise as a monotone limit. 
(Note that this monotonicity property does not hold for aBMs, since adding another annihilating Brownian motion by going from $\bfx_n$ to $\bfx_{n+1}$ might kill a previous one.)
Having constructed the infinite coalescing system, define for $y\in\bfY_t^\bfx$  
\[C(t,y):=\#\{x \in \bfx : \text{there exists a coalescing path from }(0,x) \text{ to } (t,y)\} \]
and let
\[\bfX^\bfx_t:=\{y\in\bfY^\bfx_t\,|\,C(t,y)\text{ is odd}\}.\]
Then it is easy to see that $C(t,y)$ is almost surely finite and that $(\bfX^\bfx_t)_{t\ge0}$ defines a system of aBMs starting from $\bfx$, see \cite[Lemma 5.15]{HOV15}. 
Moreover, we have $\bfX_t^{\bfx_n}\to\bfX_t^\bfx$ \emph{pathwise} almost surely, although the limit is not monotone w.r.t. set inclusion.

\subsection{Technical lemmas}\label{sec:technicalities}
\begin{lemma}\label{lem:compact}
The space $\cM_1(\R)$ with the topology of vague convergence is metrizable and compact.
\end{lemma}

\begin{proof}
We clearly have $\cM_1(\R)\subseteq B$, where $B$ denotes the closed unit ball in $L^\infty(\R)$.
It is well known that the latter space is (isometrically isomorphic to) the topological dual of $L^1(\R)$,
and easy to see that vague convergence on $B$ is equivalent to convergence w.r.t.\ the weak-$*$-topology of $L^\infty(\R)=L^1(\R)^*$ restricted to $B$.
Moreover, it is easy to check that $\cM_1(\R)$ is vaguely closed in $B$. 
As a consequence of the Banach-Alaoglu theorem (see e.g.\ \cite[Cor.\ V.4.3]{DS58}, $\cM_1(\R)$ is compact in the weak-$*$-topology, hence vaguely compact.
Moreover, by \cite[Thm.\ V.5.1]{DS58} it is also metrizable.
\end{proof}

\begin{lemma}\label{lem:dense}
$\cM_1^d(\R)$ is dense in $\cM_1(\R)$.
\end{lemma}
\begin{proof}
First consider $u(\cdot)\equiv\lambda\in(0,1)$ constant. 
Define $u^\ssup{n}\in\cM_1^d(\R)$ by 
\[u^\ssup{n}:=\sum_{k\in\Z}\ind_{[\frac{k}{n},\frac{k+\lambda}{n}]},\]
so that the interface is a translation invariant lattice $\cI(u^\ssup{n})=\frac{1}{n}(\Z+\{0,\lambda\})\in\cD$. 
Then we have
\[\langle u^\ssup{n},\phi\rangle=\sum_{k\in\Z}\int_{\frac{k}{n}}^{\frac{k+\lambda}{n}}\phi(x)\,dx\to\lambda\int_\R\phi(x)\,dx=\lambda\langle u,\phi\rangle\]
for all $\phi\in\cC_c(\R)$, i.e.\ $u^\ssup{n}\to u$ in the topology of $\cM_1(\R)$. 
By an analogous construction on compact intervals $[a,b]\subseteq\R$
and linearity, we see that we can approximate any step function
\[u=\sum_{j=1}^N\lambda_j\ind_{[a_j,b_j]},\qquad \lambda_j\in(0,1),\; a_j<b_j\]
by elements of $\cM_1^d(\R)$.
Write $\cT(\R)$ for the space of all step functions on $\R$. 
Now suppose that $u\in\cM_1(\R)$ is integrable. 
Since $\cT(\R)$ is dense in $L^1(\R)$ (w.r.t.\ the $L^1$-norm, see e.g.\ \cite[Thm.\ 1.18]{LL01}), 
we conclude that any such $u$ can be approximated by step functions in the topology of vague convergence. 
Finally, if $u\in\cM_1(\R)$ is arbitrary, we define $u^\ssup{K}:=u\ind_{[-K,K]}\in\cM_1(\R)\cap L^1(\R)$ so that $u^\ssup{K}\to u$ vaguely as $K\to\infty$. 
We have thus shown that
\[\cM_1^d(\R)\subseteq \cT(\R)\cap\cM_1(\R)\subseteq L^1(\R)\cap\cM_1(\R)\subseteq\cM_1(\R),\]
where each inclusion is dense w.r.t.\ the topology of vague convergence on $\cM_1(\R)$.  
This establishes the assertion of the lemma.
\end{proof}

{\bf Acknowledgments.} 
This project received financial support
by the German Research Foundation (DFG) within
the DFG Priority Programme 1590 `Probabilistic Structures in
Evolution', grant OR 310/1-1.

 \bibliographystyle{alpha}
 \bibliography{duality}

\end{document}